\documentclass[english,a4paper]{amsart}
\usepackage[T1]{fontenc}
\usepackage[latin9]{inputenc}
\usepackage{amsthm}
\usepackage{amsmath}
\usepackage{amssymb}
\usepackage{esint}
\usepackage{color}

\makeatletter
\theoremstyle{plain}
\newtheorem{thm}{\protect\theoremname}
  \theoremstyle{plain}
  \newtheorem{prop}[thm]{\protect\propositionname}
  \theoremstyle{plain}
  \newtheorem{lem}[thm]{\protect\lemmaname}
  \theoremstyle{remark}
  \newtheorem{rem}[thm]{\protect\remarkname}
  \theoremstyle{definition}
  \newtheorem{defn}[thm]{\protect\definitionname}

\DeclareMathOperator{\ric}{Ric_{\eta}}

\makeatother

\usepackage{babel}
  \providecommand{\definitionname}{Definition}
  \providecommand{\lemmaname}{Lemma}
  \providecommand{\propositionname}{Proposition}
  \providecommand{\remarkname}{Remark}
\providecommand{\theoremname}{Theorem}

\begin{document}

\title{Linear perturbation of the Yamabe problem on   manifolds
with boundary}

\author{Marco Ghimenti}
\address[Marco Ghimenti]{Dipartimento di Matematica
Universit\`a di Pisa
Largo Bruno Pontecorvo 5, I - 56127 Pisa, Italy}
\email{marco.ghimenti@dma.unipi.it }
\author{Anna Maria Micheletti}
\address[Anna Maria Micheletti]{Dipartimento di Matematica
Universit\`a di Pisa
Largo Bruno Pontecorvo 5, I - 56127 Pisa, Italy}
\email{a.micheletti@dma.unipi.it }

\author{Angela Pistoia}
\address[Angela Pistoia] {Dipartimento SBAI, Universt\`{a} di Roma ``La Sapienza", via Antonio Scarpa 16, 00161 Roma, Italy}
\email{angela.pistoia@uniroma1.it}

\thanks{The research that leads to the present paper was partially supported by  the group GNAMPA of Istituto Nazionale di Alta Matematica (INdAM). The first author is also partially supported by P.R.A., University of Pisa}

\begin{abstract}
We build blowing-up solutions for linear perturbation of the Yamabe problem on manifolds with boundary, provided the dimension of the manifold is $n\ge7$ and the trace-free part of the second fundamental form is non-zero everywhere on the boundary.

 \end{abstract}



\maketitle

\section{Introduction}

Given $(M,g)$   a smooth compact Riemannian manifold without boundary, the Yamabe problem is to find, in the conformal class of $g$, a metric of constant scalar curvature. The geometric problem has a PDE formulation, i.e. the metric $\tilde g=u^{4\over n-2}g$ has the required properties if the  function $u$ is a smooth positive solution to the critical equation
\begin{equation}
L_{g}u=\kappa u^{n+2\over n-2}\ \hbox{in }M, \label{yamabe}
\end{equation}
for some constant $\kappa$. 
Here   $L_{g}:=\Delta_{g}-\frac{n-2}{4(n-1)}R_{g}$ is  the conformal Laplacian,
$\Delta_{g}$ is
the Laplace Beltrami operator and $R_{g}$ is the scalar curvature of $(M,g).$
Solutions to \eqref{yamabe} are critical points of the functional
$$E(u):={\int\limits_M  \left(|\nabla u|^2+{n-2\over4(n-1)}R_gu^2\right)dv_g\over\left(\int\limits_{\partial\Omega}|u|^{2n\over n-2}d\sigma\right)^{n-2\over n}}, \ u\in H^1_g(M),$$
were $dv_g$  denotes the volume form  on $M$ and $\partial M.$
The exponent $2n\over n-2$ is critical for the Sobolev embedding $H^1_g(M)\hookrightarrow L^{2n\over n-2}(\partial M) .$  
The existence of a minimizing  solution to the Yamabe problem is well-known and follows
from the combined works of Yamabe \cite{yam},  Trudinger \cite{tru}, Aubin \cite{aub}  and Schoen \cite{sch}.
\\

One of the generalizations of this problem on manifolds $(M,g)$ with boundary was proposed by Escobar in \cite{E92} and it  consists of finding in the conformal class of $g$, a scalar-flat metric of constant boundary mean curvature. Also in this case the geometric problem has a PDE formulation, i.e.
the metric $\tilde g=u^{4\over n-2}g$ has the required properties if the  function $u$ is a smooth positive solution to the critical boundary value problem
\begin{equation}
\left\{\begin{aligned}
&L_{g}u=0\ \hbox{in }M\\
&{\partial_\nu}u+\frac{n-2}{2}H_{g}u=\kappa u^{\frac{2(n-1)}{n-2}-1} \ \hbox{on}\ \partial M.
\end{aligned}\right.\label{eq:probK}
\end{equation}
for some constant $\kappa$. 
Here   $\nu$ is the outward unit normal
vector to $\partial M$ and
$H_{g}$ is the mean curvature on $\partial M$ with respect to $g.$

Solutions to \eqref{eq:probK} are critical points of the functional
$$Q(u):={\int\limits_M  \left(|\nabla u|^2+{n-2\over4(n-1)}R_gu^2\right)dv_g+\int\limits_{\partial\Omega}{n-2\over2}H_g u^2d\sigma_g\over\left(\int\limits_{\partial\Omega}|u|^{2(n-1)\over n-2}d\sigma\right)^{n-2\over n-1}}, \ u\in H$$
were $dv_g$ and $d\sigma_g$ denote the volume forms on $M$ and $\partial M,$ respectively, and the space
$$H:=\left\{u\in H^1_g(M)\ :\ u\not=0\ \hbox{on}\ \partial\Omega\right\}.$$

Escobar in \cite{E92} introduced the Sobolev quotient
\begin{equation}\label{sob-quo}
Q(M,\partial M):=\inf\limits_H Q(u),
\end{equation}
which  is  conformally invariant and  always satisfies
\begin{equation}\label{ineq}
Q(M,\partial M)\le Q(\mathbb B^n,\partial \mathbb  B^n),
\end{equation}
where $\mathbb B^n$ is the unit ball in $\mathbb R^n$ endowed with the euclidean metric $\mathfrak g_0$.

Following Aubin's approach (see \cite{aub}), Escobar proved that if $Q(M,\partial M)$ is finite and the strict inequality in \eqref{ineq} holds, i.e.
\begin{equation}\label{strict}
Q(M,\partial M)< Q(\mathbb B^n,\partial \mathbb  B^n),
\end{equation}
then the infimum \eqref{sob-quo} is achieved and a solution to problem \eqref{eq:probK} does exist.

In the negative case, i.e. $Q(M,\partial M)\le0$, it is clear that \eqref{strict} holds.
The positive case, i.e. $Q(M,\partial M)>0$, is the most difficult one and the proof of the validity of \eqref{strict} has required a lot of works.
Assume $(M,g)$ is not conformally equivalent to $(\mathbb B^n,\mathfrak g_0)$, \eqref{strict} has been proved by Escobar in \cite{E92} if
\begin{itemize}
\item[$\diamond$] $n=3,$
\item[$\diamond$]  $n=4,5$ and $\partial M$ is umbilic,
\item[$\diamond$]  $n\ge6$, $\partial M$ is umbilic and $M$ is locally conformally flat 
\item[$\diamond$]  $n\ge6$ and $M$ has a non-umbilic point
\end{itemize}
by Marques in \cite{Ma1,Ma2} if
\begin{itemize}
\item[$\diamond$]  $n=4,5$ and $\partial M$ is not umbilic,
\item[$\diamond$]  $n\ge8$, $\overline{\textrm Weyl}_g(\xi)\not=0$ for some $\xi\in\partial M$
\item[$\diamond$]  $n\ge9$, ${\textrm Weyl}_g(\xi)\not=0$ for some $\xi\in\partial M$
\end{itemize}
by  Almaraz in \cite{A1} if
\begin{itemize}
\item[$\diamond$]  $n=6,7,8$, $\partial M$ is umbilic and $ {\textrm Weyl}_g(\xi)\not=0$ for some $\xi\in\partial M.$
\end{itemize}
 We   remind that
a point $\xi\in\partial M$ is said to be {\it umbilic} if the tensor
 $T_{ij}=h_{ij}-H_g g_{ij}$ vanishes at $\xi,$ where $h_{ij}$ are the coefficients of the second fundamental form and $H={1\over n}g^{ij}h_{ij}$ is the mean curvature.
 The boundary $\partial M$ is said to be umbilic if all its points are umbilic.
Moreover, $\overline{\textrm Weyl}_g(\xi)$ denotes the Weyl tensor  of the restriction of the metric to the boundary.

The strategy to prove that the strict inequality \eqref{strict} holds consists in finding good test functions, which  involve the minimizer of the Sobolev quotient in $\mathbb R^n_+:=\left\{(x,t)\ :\ x\in\mathbb R^{n-1},\ t>0\right\}, $ namely the so-called {\it bubble}
\begin{equation}\label{bubble}
U_{\delta,y}(x,t):=\delta^{-{n-2\over2}}U\left({x-y\over\delta},{t\over\delta}\right),\ \delta>0,\ x,y\in\mathbb R^{n-1},\ t>0\end{equation}
where
\begin{equation}\label{limite}  U(x,t):={1\over\left((1+t)^2+|x|^2\right)^{n-2\over2}}.
\end{equation}
Indeed Beckner in \cite{B} and Escobar \cite{E88} proved that
$$Q(\mathbb B^n,\partial \mathbb  B^n)=\inf\left\{{\int\limits_{\mathbb R^n_+}   |\nabla u|^2dx \over\left(\int\limits_{\partial\mathbb R^n_+}|u|^{2(n-1)\over n-2}dx\right)^{n-2\over n-1}}\ :\ u\in H^1(\mathbb R^n_+),\ u\not=0\ \hbox{on}\ \partial \mathbb R^n_+\right\}.$$ The infimum is achieved by the functions
$U_{\delta,y}$ which are the only positive solutions to the limit  problem
\begin{equation}\label{limite1}
\left\{\begin{aligned}&\Delta u=0\ \hbox{in}\ \mathbb R^n_+\\
&\partial_\nu u=(n-2) u^{n\over n-2}.\\
\end{aligned}\right.
\end{equation} 
\\

Once the existence of solutions of problems \eqref{yamabe} or (\ref{eq:probK}) is settled, a natural question  concerns the structure of 
the  full set of positive solutions of \eqref{yamabe} or (\ref{eq:probK}).
Concerning the Yamabe problem on manifold without boundary,   Schoen (see \cite{s3}) 
raised the question of compactness
of the   set of solutions of problem \eqref{yamabe}. The question has been 
recently resolved by S. Brendle, M. A. Khuri, F. C. Marques and R. Schoen in a series of works \cite{bre,bre-mar,khu-mar-sch} (see also the survey by Marques \cite{mar}).
By their results, the set of solutions for the Yamabe problem \eqref{yamabe} is compact on any compact manifold of dimension $n\le24$, while it is not compact on some compact manifold of dimension $n\ge25.$

Therefore, it is natural to address the question of compactness of the set of positive solutions of (\ref{eq:probK}). 
If $Q(M,\partial M)<0$ the solution is unique and if $Q(M,\partial M)=0$ the solution is unique up to a constant factor.
If $Q(M,\partial M)>0$ the situation turns out to be more delicate. Indeed in  the case of the euclidean ball $(\mathbb B^n,\mathfrak g_0)$ the set of solutions is not compact! 
 Felli and Ould-Ahmedou \cite{FO03} proved that  compactness holds 
when $n\ge3$, $(M,g)$ is locally conformally flat and $\partial M$ is umbilic.
Almaraz in \cite{A3} proved that compactness also holds  if $n\ge7$ and the trace-free second fundamental form of $\partial M$ is non
zero everywhere. This last assumption is generic as a  transversality argument shows. 
Up to our knowledge, the only non-compactness result is due to  Almaraz.  In \cite{A2}  he   constructs a sequence of blowing-up conformal metrics with zero scalar curvature and constant boundary mean curvature on a ball of dimension $n\ge 25.$
It is unknown if the dimension $25$ is sharp for the compactness, namely if $n\le24$ the problem (\ref{eq:probK}) is compact or not.\\

In this paper we are interested in the existence of blowing-up solutions to problems which are linear perturbation of the geometric problem \eqref{eq:probK}.
More precisely, the question we address is the following.
{\it Does the problem
\begin{equation}\label{eq:Peps}
\left\{\begin{aligned}
&L_{g}u=0\ \hbox{in }M\\
&{\partial_\nu}u+\frac{n-2}{2}H_{g}u+\varepsilon \gamma u=  u^{ \frac{2(n-1)}{n-2}-1} \ \hbox{on}\ \partial M.
\end{aligned}\right.
\end{equation}
where $\gamma\in C^2(M)$, have  positive blowing-up solutions   as the positive parameter  $\varepsilon$ approaches zero?}

We give a positive answer under suitable geometric assumptions on $M$ and on the sign of the linear perturbation term $\gamma.$ Our main result reads as follows.
\begin{thm}
\label{thm:main} Assume $n\ge7$, $Q(M,\partial M)>0$ and
the trace-free second fundamental form of $\partial M$ is non zero everywhere.
If the function $\gamma\in C^{1}(M)$ is strictly positive, then for
$\varepsilon>0$ small there exists a positive solution $u_{\varepsilon}$
of (\ref{eq:Peps}) such that $\|u_{\varepsilon}\|_{H^{1}}$ is bounded
and $u_{\varepsilon}$ blows-up at a suitable point $q_{0}\in\partial M$
as $\varepsilon\rightarrow0.$ 
\end{thm}

\begin{rem}

The proof of our result relies on a Ljapunov-Schmidt procedure. We build solutions to \eqref{eq:Peps} which at the main  order   looks like the bubble \eqref{bubble} centered at a point $q_0$ on the boundary. 
As usual the blowing-up point $q_0$  turns out to be a critical point of the reduced energy whose leading term is a function (see \eqref{ridotta}) defined on the boundary,
which cannot be explicitly written in terms of the geometry quantities of the boundary.  The  difficulty comes from the fact that we cannot  find an explicit expression of the correction term we need to add
to the bubble to have a good approximation. The correction term  solves the linear problem \eqref{eq:vqdef}  and it gives a significant contribution to the reduced energy (see 	\eqref{phiq}).
Actually, we conjecture that the term \eqref{phiq} (up to a constant factor) is nothing but the trace-free second fundamental form at $q_0$ and  so the blowing-up point $q_0$ is  a critical point of the function
$$q\to{\|\hbox{the trace-free second fundamental form at $q$}\|\over\gamma^2(q)},\ q\in\partial M.$$

\end{rem}

\begin{rem}
Theorem \ref{thm:main} states that problem  (\ref{eq:Peps})  is not compact if the linear perturbation term is strictly positive in $\partial M.$ We strongly believe that the compactness is recovered if the linear perturbation is negative somewhere in $\partial\Omega$. This is what happens in the case of   linear perturbation of the  Yamabe problem \eqref{yamabe}.
Indeed, if we consider the perturbed problem
\begin{equation}
L_{g}u+\varepsilon f u=\kappa u^{n+2\over n-2}\ \hbox{in }M, \label{per-yamabe}
\end{equation}
where $\varepsilon$ is a positive parameter and $f\in C^2(M)$.
Druet in \cite{dru} shows that if $f\le0$ in $M,$ blow-up does not occur if $3\le n\le5.$
 When $f$ is positive somewhere in $M$, blow-up is possible as showed by  Druet and Hebey in \cite{dru-heb} in the case of the sphere
 and by Esposito, Pistoia, and V\'etois  in \cite{EPV14} on general compact manifolds.
 \end{rem}

\begin{rem}
 Almaraz in \cite{A3}   studied the compactness of problem \eqref{eq:probK} when the exponent in the non-linearity of the boundary is below the critical exponent and he proved the following result.
\begin{thm}\label{almaraz} Assume $n\ge7$, $Q(M,\partial M)>0$ and
the trace-free second fundamental form of $\partial M$ is non zero everywhere.
Then the problem
\begin{equation}\label{peps}
\left\{\begin{aligned}
&L_{g}u=0\ \hbox{in }M\\
&{\partial_\nu}u+\frac{n-2}{2}H_{g}u=  u^{ \frac{2(n-1)}{n-2}-1-\varepsilon} \ \hbox{on}\ \partial M.
\end{aligned}\right.
\end{equation}
is compact, namely there exist $\varepsilon_0>0$ and   a positive constant $C$ such that   for any $\varepsilon\in(0,\varepsilon_0)$ any positive solution $u_\varepsilon$ of \eqref{peps} satisfies
$\|u_\varepsilon\|_{C^{2,\alpha}(M)}\le C$ for some $\alpha\in(0,1).$\end{thm} In other words, problem \eqref{peps}
 does not have any blowing-up solutions as the positive parameter $\varepsilon$ approaches zero. 
Let us point out that combining our argument with some ideas developed in a previous paper \cite{GMP16} we can also obtain the existence of blowing-up solutions for problem \eqref{peps} when the parameter 
$\varepsilon$ is {\it negative} and small. Then the  compactness result Theorem \ref{almaraz} is  sharp, namely the problem \eqref{peps} is compact 
 if the exponent in the non-linearity of the boundary approaches the critical exponent  from below and it is non-compact   if the exponent approaches the critical exponent from above.
\end{rem}

The paper is organized as follows. In Section \ref{uno} we set the problem in a suitable scheme, in Section \ref{due} we perform the finite-dimensional reduction, in Section \ref{tre} we study the reduced problem and in Section \ref{quattro} we prove Theorem \ref{thm:main}. The Appendix contains some technical results.

\section{Variational framework and preliminaries}\label{uno}

It is well known \cite{E92} that there exists a global conformal transformation
which maps the manifold $M$ in a manifold for which the mean curvature
of the boundary is identically zero, so we can choose a metric $(M,g)$ such
that $H_{g}\equiv0$. This can be done, by a global conformal transformation
$g=\varphi_{1}^{4/n}\bar{g}$, where $\varphi_{1}$ is the positive
eigenvector of the first eigenvalue $\lambda_{1}$ of the problem
\[
\left\{ \begin{array}{ccc}
-L_{g}\varphi+\lambda_{1}\varphi=0 &  & \text{on }M;\\
B_{g}\varphi=0 &  & \text{on \ensuremath{\partial}}M.
\end{array}\right.
\]
It is useful to point out that if $\pi$ denotes the second fundamental form related to $g$ and $q\in\partial M$
then $\pi(q)$ is non-zero if and only if the trace-free second fundamental form related to $\bar g$ at the point $q$ is non-zero.

By the assumption $Q(M,\partial M)>0$ we have $K>0$ in (\ref{eq:probK}),
so we can normalize it to be $(n-2)$. Moreover, to gain in readability,
we set $a=\frac{n-2}{4(n-1)}R_{g}$ , so Problem (\ref{eq:Peps})
reads as

\begin{equation}
\left\{ \begin{array}{ccc}
-\Delta_{g}u+au=0 &  & \text{on }M;\\
\frac{\partial u}{\partial\nu}+\varepsilon\gamma u=(n-2)\left(u^{+}\right)^{\frac{n}{n-2}} &  & \text{on \ensuremath{\partial}}M.
\end{array}\right.\label{eq:P}
\end{equation}
Since $Q(M,\partial M)>0$, we can endow $H^{1}(M)$ with the following
equivalent scalar product
\[
\left\langle \left\langle u,v\right\rangle \right\rangle _{H}=\int_{M}(\nabla_{g}u\nabla_{g}v+auv)d\mu_{g}
\]
which leads to the equivalent norm $\|\cdot\|_{H}$. We have the well
know maps
\begin{align*}
i: & H^{1}(M)\rightarrow L^{t}(\partial M)\\
i^{*}: & L^{t'}(\partial M)\rightarrow H^{1}(M)
\end{align*}
for $1\le t\le\frac{2(n-1)}{n-2}$ (and for $1\le t<\frac{2(n-1)}{n-2}$
the embedding $i$ is compact). 

Given $f\in L^{\frac{2(n-1)}{n-2}}(\partial M)$ there exists a unique
$u\in H^{1}(M)$ such that
\begin{align}
u=i^{*}(f) & \iff\left\langle \left\langle u,\varphi\right\rangle \right\rangle _{H}=\int_{\partial M}f\varphi d\sigma\text{ for all }\varphi\nonumber \\
 & \iff\left\{ \begin{array}{ccc}
-\Delta_{g}u+au=0 &  & \text{on }M;\\
\frac{\partial u}{\partial\nu}=f &  & \text{on \ensuremath{\partial}}M.
\end{array}\right.\label{eq:istella}
\end{align}
The functional defined on $H^{1}(M)$ associated to (\ref{eq:P})
is
\[
J_{\varepsilon}(u):=\frac{1}{2}\int_{M}|\nabla_{g}u|^{2}+au^{2}d\mu_{g}+\frac{1}{2}\int_{\partial M}\varepsilon\gamma u^{2}d\sigma-\frac{(n-2)^{2}}{2(n-1)}\int_{\partial M}\left(u^{+}\right)^{\frac{2(n-1)}{n-2}}d\sigma.
\]
 To solve problem (\ref{eq:P}) is equivalent to find $u\in H^{1}(M)$
such that
\begin{equation}
u=i^{*}(f(u)-\varepsilon\gamma u)\label{eq:P*}
\end{equation}
 where $f(u)=(n-2)\left(u^{+}\right)^{\frac{n}{n-2}}$. We remark
that, if $u\in H^{1}$, then $f(u)\in L^{\frac{2(n-1)}{n}}(\partial M)$. 

Given $q\in\partial M$ and $\psi_{q}^{\partial}:\mathbb{R}_{+}^{n}\rightarrow M$
the Fermi coordinates in a neighborhood of $q$; we define
\begin{align*}
W_{\delta,q}(\xi) & =U_{\delta}\left(\left(\psi_{q}^{\partial}\right)^{-1}(\xi)\right)\chi\left(\left(\psi_{q}^{\partial}\right)^{-1}(\xi)\right)=\\
 & =\frac{1}{\delta^{\frac{n-2}{2}}}U\left(\frac{y}{\delta}\right)\chi(y)=\frac{1}{\delta^{\frac{n-2}{2}}}U\left(x\right)\chi(\delta x)
\end{align*}
where $y=(z,t)$, with $z\in\mathbb{R}^{n-1}$ and $t\ge0$, $\delta x=y=\left(\psi_{q}^{\partial}\right)^{-1}(\xi)$
and $\chi$ is a radial cut off function, with support in ball of
radius $R$. 

Here $U_{\delta}(y)=\frac{1}{\delta^{\frac{n-2}{2}}}U\left(\frac{y}{\delta}\right)$
is the one parameter family of solution of the problem
\begin{equation}
\left\{ \begin{array}{ccc}
-\Delta U_{\delta}=0 &  & \text{on }\mathbb{R}_{+}^{n};\\
\frac{\partial U_{\delta}}{\partial t}=-(n-2)U_{\delta}^{\frac{n}{n-2}} &  & \text{on \ensuremath{\partial}}\mathbb{R}_{+}^{n}.
\end{array}\right.\label{eq:Udelta}
\end{equation}
and ${\displaystyle U(z,t):=\frac{1}{\left[(1+t)^{2}+|z|^{2}\right]^{\frac{n-2}{2}}}}$
is the standard bubble in $\mathbb{R}_{+}^{n}$.

Moreover, we consider the functions 
\begin{eqnarray*}
j_{i}=\frac{\partial U}{\partial x_{i}},\ i=1,\dots n-1 &  & j_{n}=\frac{n-2}{2}U+\sum_{i=1}^{n}y_{i}\frac{\partial U}{\partial y_{i}}
\end{eqnarray*}
which are solutions of the linearized problem 
\begin{equation}
\left\{ \begin{array}{ccc}
-\Delta\phi=0 &  & \text{on }\mathbb{R}_{+}^{n};\\
\frac{\partial\phi}{\partial t}+nU^{\frac{2}{n-2}}\phi=0 &  & \text{on \ensuremath{\partial}}\mathbb{R}_{+}^{n}.
\end{array}\right.\label{eq:linearizzato}
\end{equation}
Given $q\in\partial M$ we define, for $b=1,\dots,n$
\[
Z_{\delta,q}^{b}(\xi)=\frac{1}{\delta^{\frac{n-2}{2}}}j_{b}\left(\frac{1}{\delta}\left(\psi_{q}^{\partial}\right)^{-1}(\xi)\right)\chi\left(\left(\psi_{q}^{\partial}\right)^{-1}(\xi)\right)
\]
and we decompose $H^{1}(M)$ in the direct sum of the following two
subspaces
\begin{align*}
K_{\delta,q} & =\text{Span}\left\langle Z_{\delta,q}^{1},\dots,Z_{\delta,q}^{n}\right\rangle \\
K_{\delta,q}^{\bot} & =\left\{ \varphi\in H^{1}(M)\ :\ \left\langle \left\langle \varphi,Z_{\delta,q}^{b}\right\rangle \right\rangle _{H}=0,\ b=1,\dots,n\right\} 
\end{align*}
and we define the projections
\begin{eqnarray*}
\Pi=H^{1}(M)\rightarrow K_{\delta,q} &  & \Pi^{\bot}=H^{1}(M)\rightarrow K_{\delta,q}^{\bot}.
\end{eqnarray*}

Given $q\in\partial M$ we also define in a similar way
\[
V_{\delta,q}(\xi)=\frac{1}{\delta^{\frac{n-2}{2}}}v_{q}\left(\frac{1}{\delta}\left(\psi_{q}^{\partial}\right)^{-1}(\xi)\right)\chi\left(\left(\psi_{q}^{\partial}\right)^{-1}(\xi)\right),
\]
and 
\begin{equation}
\left(v_{q}\right)_{\delta}(y)=\frac{1}{\delta^{\frac{n-2}{2}}}v_{q}\left(\frac{y}{\delta}\right);\label{eq:vqdelta}
\end{equation}
here $v_{q}:\mathbb{R}_{+}^{n}\rightarrow\mathbb{R}$ is the unique
solution of the problem 
\begin{equation}
\left\{ \begin{array}{ccc}
-\Delta v=2h_{ij}(q)t\partial_{ij}^{2}U &  & \text{on }\mathbb{R}_{+}^{n};\\
\frac{\partial v}{\partial t}+nU^{\frac{2}{n-2}}v=0 &  & \text{on \ensuremath{\partial}}\mathbb{R}_{+}^{n}.
\end{array}\right.\label{eq:vqdef}
\end{equation}
such that $v_{q}$ is $L^{2}(\mathbb{R}_{+}^{n})$-ortogonal to $j_{b}$
for all $b=1,\dots,n$ Here $h_{ij}$ is the second fundamental form
and we use the Einstein convention of repeated indices. We remark
\begin{equation}
|\nabla^{r}v_{q}(y)|\le C(1+|y|)^{3-r-n}\text{ for }r=0,1,2,\label{eq:gradvq}
\end{equation}
\begin{equation}
\int_{\partial\mathbb{R}_{+}^{n}}U^{\frac{n}{n-2}}v_{q}=0\label{eq:Uvq}
\end{equation}
and
\begin{equation}\label{new}
\int_{\partial\mathbb{R}_{+}^{n}}\Delta v_{q}v_qdzdt\le 0,
\end{equation}
(see \cite[Proposition 5.1 and estimate (5.9)]{A3}).
\begin{prop}
The map $q\mapsto v_{q}$ is in $C^{2}(\partial M)$.\end{prop}
\begin{proof}
Let $q_{0}\in\partial M$. If $q\in\partial M$ is sufficiently close
to $q_{0}$, in Fermi coordinates we have $q=q(y)=\exp_{q_{0}}y$,
with $y\in\mathbb{R}^{n-1}$. So $v_{q}=v_{\exp_{q_{0}}y}$ and we
define
\[
\Gamma_{i}=\left.\frac{\partial}{\partial y_{i}}v_{\exp_{q_{0}}y}\right|_{y=0}.
\]
We prove the result for $\Gamma_{1}$, being the other cases completely
analogous. By (\ref{eq:vqdef}) we have that $\Gamma_{1}$ solves
\[
\left\{ \begin{array}{ccc}
-\Delta\Gamma_{1}=2\left(\left.\frac{\partial}{\partial y_{1}}\left(h_{ij}(q(y))\right)\right|_{y=0}\right)t\partial_{ij}^{2}U &  & \text{on }\mathbb{R}_{+}^{n};\\
\frac{\partial\Gamma_{1}}{\partial t}+nU^{\frac{2}{n-2}}\Gamma_{1}=0 &  & \text{on \ensuremath{\partial}}\mathbb{R}_{+}^{n}.
\end{array}\right.
\]
and, by the result of \cite{A3}, we know that $\Gamma_{1}$ exists.
We can proceed in analogous way for the second derivative. 
\end{proof}
We define the useful integral quantity 
\[
I_{m}^{\alpha}=\int_{0}^{\infty}\frac{\rho^{\alpha}}{(1+\rho^{2})^{m}}d\rho
\]
and in the appendix (Remark \ref{lem:I-a-m}) we recall some useful
estimates of these integrals.

Finally, we have to we recall the Taylor expansion for the metric
$g$ and for the volume form on $M$, expressed by the Fermi coordinates.

Since, without loss of generality, we have chosen a manifold for which
$H_{g}\equiv0$, we have the following expansions in a neighborhood
of $y=0$, with the usual notation $y=(z,t)$, where $z\in\mathbb{R}^{n}$
and $t\ge0$. Here and in the following, we use the Einstein convention
on the sum of repeated indices. Moreover, we use the convention that
$a,b,c,d=1,\dots,n$ and $i,j,k,l=1,\dots,n-1$.
\begin{align}
|g(y)|^{1/2}= & 1-\frac{1}{2}\left[\|\pi\|^{2}+\ric(0)\right]t^{2}-\frac{1}{6}\bar{R}_{ij}(0)z_{i}z_{j}+O(|y|^{3})\label{eq:|g|}\\
g^{ij}(y)= & \delta_{ij}+2h_{ij}(0)t+\frac{1}{3}\bar{R}_{ikjl}(0)z_{k}z_{l}+2\frac{\partial h_{ij}}{\partial z_{k}}(0)tz_{k}\nonumber \\
 & +\left[R_{injn}(0)+3h_{ik}(0)h_{kj}(0)\right]t^{2}+O(|y|^{3})\label{eq:gij}\\
g^{an}(y)= & \delta_{an}\label{eq:gin}
\end{align}
where $\pi$ is the second fundamental form and $h_{ij}(0)$ are its
coefficients, $\bar{R}_{ikjl}(0)$ and $R_{abcd}(0)$ are the curvature
tensor of $\partial M$ and $M$, respectively, $\bar{R}_{ij}(0)=\bar{R}_{ikjk}(0)$
are the coefficients of the Ricci tensor, and $\ric(0)=R_{nini}(0)=R_{nn}(0)$
(see \cite{E92}).

\section{Finite dimensional reduction}\label{due}

We look for a good approximation for the solution of problem (\ref{eq:P*}),
then we look for solution with the form
\[
u=W_{\delta,q}+\delta V_{\delta,q}+\Phi,\text{ with }\Phi\in K_{\delta,q}^{\bot}.
\]
and we project (\ref{eq:P*}) on $K_{\delta,q}^{\bot}$ and $K_{\delta,q}$
obtaining 
\begin{align}
\Pi^{\bot}\left\{ W_{\delta,q}+\delta V_{\delta,q}+\Phi-i^{*}\left(f(W_{\delta,q}+\delta V_{\delta,q}+\Phi)-\varepsilon\gamma(W_{\delta,q}+\delta V_{\delta,q}+\Phi)\right)\right\}  & =0;\label{eq:P-Kort}\\
\Pi\left\{ W_{\delta,q}+\delta V_{\delta,q}+\Phi-i^{*}\left(f(W_{\delta,q}+\delta V_{\delta,q}+\Phi)-\varepsilon\gamma(W_{\delta,q}+\delta V_{\delta,q}+\Phi)\right)\right\}  & =0.\label{eq:P-K}
\end{align}
To solve (\ref{eq:P-Kort}) we define the linear operator $L=L_{\delta,q}:K_{\delta,q}^{\bot}\rightarrow K_{\delta,q}^{\bot}$
as
\begin{equation}
L(\Phi)=\Pi^{\bot}\left\{ \Phi-i^{*}\left(f'(W_{\delta,q}+\delta V_{\delta,q})[\Phi]\right)\right\} \label{eq:defL}
\end{equation}
 and a nonlinear term $N(\Phi)$ and a remainder term $R$ as
\begin{align}
N(\Phi)= & \Pi^{\bot}\left\{ i^{*}\left(f(W_{\delta,q}+\delta V_{\delta,q}+\Phi)-f(W_{\delta,q}+\delta V_{\delta,q})-f'(W_{\delta,q}+\delta V_{\delta,q})[\Phi]\right)\right\} \label{eq:defN}\\
R= & \Pi^{\bot}\left\{ i^{*}\left(f(W_{\delta,q}+\delta V_{\delta,q})\right)-W_{\delta,q}-\delta V_{\delta,q}\right\} \label{eq:defR}
\end{align}
so eq (\ref{eq:P-Kort}) rewrites as 
\[
L(\Phi)=N(\Phi)+R-\Pi^{\bot}\left\{ i^{*}\left(\varepsilon\gamma(W_{\delta,q}+\delta V_{\delta,q}+\Phi)\right)\right\} .
\]

\begin{lem}
\label{prop:L}Let $\delta=\varepsilon\lambda$ For $a,b\in\mathbb{R}$,
$0<a<b$ there exists a positive constant $C=C(a,b)$ such that, for
$\varepsilon$ small, for any $q\in\partial M$, for any $\lambda\in[a,b]$
and for any $\phi\in K_{\delta,q}^{\bot}$ there holds
\[
\|L_{\delta,q}(\phi)\|_{H}\ge C\|\phi\|_{H}.
\]

\end{lem}
The proof of this lemma is postponed in the appendix
\begin{lem}
\label{lem:R}Assume $n\ge7$ and $\delta=\lambda\varepsilon$, then
it holds 
\[
\|R\|_{H}=O\left(\varepsilon^{2}\right)
\]
$C^{0}$-uniformly for $q\in\partial M$ and $\lambda$ in a compact
set of $(0,+\infty)$.\end{lem}
\begin{proof}
We recall that there is a unique $\Gamma$ such that 
\[
\Gamma=i^{*}\left(f(W_{\delta,q}+\delta V_{\delta,q})\right),
\]
that is, according to (\ref{eq:istella}) equivalent to say that there
exists a unique $\Gamma$ solving
\[
\left\{ \begin{array}{ccc}
-\Delta_{g}\Gamma+a\Gamma=0 &  & \text{on }M;\\
\frac{\partial\Gamma}{\partial\nu}=(n-2)\left((W_{\delta,q}+\delta V_{\delta,q})^{+}\right)^{\frac{n}{n-2}} &  & \text{on \ensuremath{\partial}}M.
\end{array}\right.
\]
By definition of $i^{*}$ we have that
\begin{align*}
\|R\|_{H}^{2}= & \|\Gamma-W_{\delta,q}-\delta V_{\delta,q}\|_{H}^{2}\\
= & \int_{M}\left[-\Delta_{g}(\Gamma-W_{\delta,q}-\delta V_{\delta,q})+a(\Gamma-W_{\delta,q}-\delta V_{\delta,q})\right](\Gamma-W_{\delta,q}-\delta V_{\delta,q})d\mu_{g}\\
 & +\int_{\partial M}\left[\frac{\partial}{\partial\nu}(\Gamma-W_{\delta,q}-\delta V_{\delta,q})\right](\Gamma-W_{\delta,q}-\delta V_{\delta,q})d\sigma\\
= & \int_{M}\left[\Delta_{g}(W_{\delta,q}+\delta V_{\delta,q})-a(W_{\delta,q}+\delta V_{\delta,q})\right]Rd\mu_{g}\\
 & \int_{\partial M}\left[(n-2)\left((W_{\delta,q}+\delta V_{\delta,q})^{+}\right)^{\frac{n}{n-2}}-\frac{\partial}{\partial\nu}(W_{\delta,q}+\delta V_{\delta,q})\right]Rd\sigma
\end{align*}
We have
\begin{equation}
\int_{M}aW_{\delta,q}Rd\mu_{g}\le c\|W_{\delta,q}\|_{L^{\frac{2n}{n+2}}(M)}\|R\|_{L^{\frac{2n}{n-2}}(M)}\le c\delta^{2}\|U\|_{L^{\frac{2n}{n+2}}(\mathbb{R}^{n})}\|R\|_{H}\label{eq:WR}
\end{equation}
and $\|U\|_{L^{\frac{2n}{n+2}}(\mathbb{R}^{n})}$ is bounded since
$n>6$. Moreover
\begin{equation}
\delta\int_{M}aV_{\delta,q}Rd\mu_{g}\le c\delta\|V_{\delta,q}\|_{L^{2}(M)}\|R\|_{L^{2}(M)}\le c\delta^{2}\|v_{q}\|_{L^{2}(\mathbb{R}^{n})}\|R\|_{H}\label{eq:VR}
\end{equation}
and, in light of (\ref{eq:gradvq}), $\|v_{q}\|_{L^{2}(\mathbb{R}^{n})}$
is bounded since $n>6$.

We have 
\begin{align*}
\int_{\partial M}\left[(n-2)W_{\delta,q}^{\frac{n}{n-2}}-\frac{\partial}{\partial\nu}W_{\delta,q}\right]Rd\sigma & \le\left\Vert (n-2)W_{\delta,q}^{\frac{n}{n-2}}-\frac{\partial}{\partial\nu}W_{\delta,q}\right\Vert _{L^{\frac{2(n-1)}{n}}(\partial M)}\|R\|_{H}\\
 & \le c\delta^{2}\|R\|_{H}
\end{align*}
since $U$ is a solution of (\ref{eq:Udelta}). In fact
\begin{multline*}
\left\Vert (n-2)W_{\delta,q}^{\frac{n}{n-2}}-\frac{\partial}{\partial\nu}W_{\delta,q}\right\Vert _{L^{\frac{2(n-1)}{n}}(\partial M)}=\\
\left(\int_{\partial\mathbb{R}_{+}^{n}}|g(\delta z,0)|^{\frac{1}{2}}\left[(n-2)U^{\frac{n}{n-2}}(z,0)\chi^{\frac{n}{n-2}}(\delta z,0)-\chi(\delta z,0)\frac{\partial U}{\partial t}(z,0)\right]^{\frac{2(n-1)}{n}}dz\right)^{\frac{n}{2(n-1)}}\\
\le C\left(\int_{\mathbb{R}^{n-1}}\left[(n-2)U^{\frac{n}{n-2}}(z,0)\left[\chi^{\frac{n}{n-2}}(\delta z,0)-\chi(\delta z,0)\right]\right]^{\frac{2(n-1)}{n}}dz\right)^{\frac{n}{2(n-1)}}=O(\delta^{2}),
\end{multline*}
Now we estimate
\begin{multline*}
\int_{\partial M}\left\{ (n-2)\left[\left((W_{\delta,q}+\delta V_{\delta,q})^{+}\right)^{\frac{n}{n-2}}-W_{\delta,q}^{\frac{n}{n-2}}\right]-\delta\frac{\partial V_{\delta,q}}{\partial\nu}\right\} Rd\sigma\\
\le c\left\Vert (n-2)\left[\left((W_{\delta,q}+\delta V_{\delta,q})^{+}\right)^{\frac{n}{n-2}}-W_{\delta,q}^{\frac{n}{n-2}}\right]-\delta\frac{\partial V_{\delta,q}}{\partial\nu}\right\Vert _{L^{\frac{2(n-1)}{n}}(\partial M)}\|R\|_{H}
\end{multline*}
and, by Taylor expansion and by definition of the function $v_{q}$
(see (\ref{eq:vqdef}) ) 
\begin{multline*}
\left\Vert (n-2)\left[\left((W_{\delta,q}+\delta V_{\delta,q})^{+}\right)^{\frac{n}{n-2}}-W_{\delta,q}^{\frac{n}{n-2}}\right]-\delta\frac{\partial V_{\delta,q}}{\partial\nu}\right\Vert _{L^{\frac{2(n-1)}{n}}(\partial M)}\\
\le\left\Vert (n-2)\left[\left((U+\delta v_{q})^{+}\right)^{\frac{n}{n-2}}-U^{\frac{n}{n-2}}\right]+\delta\frac{\partial v_{q}}{\partial t}\right\Vert _{L^{\frac{2(n-1)}{n}}(\partial\mathbb{R}_{+}^{n})}+o(\delta^{2})\\
\le\delta\left\Vert n\left((U+\theta\delta v_{q})^{+}\right)^{\frac{2}{n-2}}v_{q}+\frac{\partial v_{q}}{\partial t}\right\Vert _{L^{\frac{2(n-1)}{n}}(\partial\mathbb{R}_{+}^{n})}+o(\delta^{2})\\
=\delta n\left\Vert \left((U+\theta\delta v_{q})^{+}\right)^{\frac{2}{n-2}}v_{q}-U^{\frac{2}{n-2}}v_{q}\right\Vert _{L^{\frac{2(n-1)}{n}}(\partial\mathbb{R}_{+}^{n})}+o(\delta^{2}).
\end{multline*}
We observe that, chosen a large positive $R$, we have $U+\theta\delta v_{q}>0$
in $B(0,R)$ for some $\delta$. Moreover, on the complementary of
this ball, we have $\frac{c}{|y|^{n-2}}\le U(y)\le\frac{C}{|y|^{n-2}}$
and $|v_{q}|\le\frac{C_{1}}{|y|^{n-3}}$ for some positive constants
$c,C,C_{1}$. So it is possible to prove that, for $\delta$ small
enough, $U+\theta\delta v_{q}>0$ if $|y|\le1/\delta$. At this point
\begin{multline*}
\int_{\partial\mathbb{R}_{+}^{n}}\left[\left|\left((U+\theta\delta v_{q})^{+}\right)^{\frac{2}{n-2}}-U^{\frac{2}{n-2}}\right||v_{q}|\right]^{\frac{2(n-1)}{n}}\\
=\int_{U+\theta\delta v_{q}>0}\left[\left|\left((U+\theta\delta v_{q})^{+}\right)^{\frac{2}{n-2}}-U^{\frac{2}{n-2}}\right||v_{q}|\right]^{\frac{2(n-1)}{n}}dz\\
+\int_{U+\theta\delta v_{q}\le0}\left[\left|\left((U+\theta\delta v_{q})^{+}\right)^{\frac{2}{n-2}}-U^{\frac{2}{n-2}}\right||v_{q}|\right]^{\frac{2(n-1)}{n}}dz\\
=\delta^{\frac{2(n-1)}{n}}\int_{U+\theta\delta v_{q}>0}\left(U+\theta_{1}\delta v_{q}\right)^{\frac{-2(n-1)(n-4)}{n(n-2)}}|v_{q}|^{\frac{4(n-1)}{n}}dz\\
+\int_{U+\theta\delta v_{q}\le0}U^{\frac{4(n-1)}{n(n-2)}}|v_{q}|^{\frac{2(n-1)}{n}}dz\\
\le\delta^{\frac{2(n-1)}{n}}\int_{U+\theta\delta v_{q}>0}\left(U+\theta_{1}\delta v_{q}\right)^{\frac{-2(n-1)(n-4)}{n(n-2)}}|v_{q}|^{\frac{4(n-1)}{n}}dz\\
+\int_{|z|>\frac{1}{\delta}}U^{\frac{4(n-1)}{n(n-2)}}|v_{q}|^{\frac{2(n-1)}{n}}dz
\end{multline*}
and, since $n>6$ one can check that $\int_{U+\theta\delta v_{q}>0}\left(U+\theta_{1}\delta v_{q}\right)^{\frac{-2(n-1)(n-4)}{n(n-2)}}|v_{q}|^{\frac{4(n-1)}{n}}dz$
is bounded and that
\begin{align*}
\int_{|z|>\frac{1}{\delta}}U^{\frac{4(n-1)}{n(n-2)}}|v_{q}|^{\frac{2(n-1)}{n}}dz & \le C\int_{|z|>\frac{1}{\delta}}\frac{1}{|z|^{\frac{4(n-1)}{n}}}\frac{1}{|z|^{\frac{2(n-1)(n-3)}{n}}}dz\\
 & \le C\int_{\frac{1}{\delta}}^{\infty}r^{-\frac{3n^{2}-12n+10}{n}}=O(\delta^{\frac{3n^{2}-11n+10}{n}})=o(\delta^{\frac{2(n-1)}{n}})
\end{align*}
thus $\left\Vert (n-2)\left[\left((W_{\delta,q}+\delta V_{\delta,q})^{+}\right)^{\frac{n}{n-2}}-W_{\delta,q}^{\frac{n}{n-2}}\right]-\delta\frac{\partial V_{\delta,q}}{\partial\nu}\right\Vert _{L^{\frac{2(n-1)}{n}}(\partial M)}=O(\delta^{2})$
and
\[
\int_{\partial M}\left\{ (n-2)\left[\left((W_{\delta,q}+\delta V_{\delta,q})^{+}\right)^{\frac{n}{n-2}}-W_{\delta,q}^{\frac{n}{n-2}}\right]-\delta\frac{\partial V_{\delta,q}}{\partial\nu}\right\} Rd\sigma\le c\delta^{2}\|R\|_{H}.
\]
To complete the proof we have to estimate 
\[
\int_{M}\left[\Delta_{g}(W_{\delta,q}+\delta V_{\delta,q})\right]Rd\mu_{g}\le\|\Delta_{g}(W_{\delta,q}+\delta V_{\delta,q})\|_{L^{\frac{2n}{n+2}}(M)}\|R\|_{H}.
\]
 We recall that in local charts the Laplace Beltrami operator is 
\begin{eqnarray*}
\Delta_{g}W_{\delta,q} & = & \Delta_{\text{euc}}\left(U_{\delta}(u)\chi(y)\right)+[g^{ij}(y)-\delta_{ij}]\partial_{ij}^{2}\left(U_{\delta}(u)\chi(y)\right)\\
 &  & -g^{ij}(y)\Gamma_{ij}^{k}(y)\partial_{k}\left(U_{\delta}(u)\chi(y)\right)
\end{eqnarray*}
where $i,k=1,\dots,n-1$, $\Delta_{\text{euc}}$ is the euclidean
Laplacian, and $\Gamma_{ij}^{k}$ are the Christoffel symbols. Notice
that, by (\ref{eq:|g|}) and (\ref{eq:gij}) we have that $\Gamma_{ij}^{k}(y)=O(|y|)$.
Now, by (\ref{eq:Udelta}) and (\ref{eq:gij}) we have, in variables
$y=\delta x$, 
\begin{eqnarray}
\Delta_{g}W_{\delta,q} & = & U_{\delta}(u)\Delta_{\text{euc}}\left(\chi(y)\right)+2\nabla U_{\delta}(u)\nabla\chi(y)\nonumber \\
 &  & +[g^{ij}(y)-\delta_{ij}]\partial_{ij}^{2}\left(U_{\delta}(u)\chi(y)\right)-g^{ij}(y)\Gamma_{ij}^{k}(y)\partial_{k}\left(U_{\delta}(u)\chi(y)\right)\nonumber \\
 & = & \frac{1}{\delta^{\frac{n-2}{2}}}\left(2h_{ij}(0)\delta x_{n}\frac{1}{\delta^{2}}\partial_{ij}U(x)+g^{ij}(x)\Gamma_{ij}^{k}(x)\frac{1}{\delta}\partial_{k}U+o(\delta)c(x)\right)\nonumber \\
 & = & \frac{1}{\delta^{\frac{n}{2}}}\left(2h_{ij}(0)x_{n}\partial_{ij}^{2}U(x)+O(\delta)c(x)\right)\label{eq:R1}
\end{eqnarray}
where, with abuse of notation, we call $c(x)$ a suitable function
such that $\left|\int_{\mathbb{R}^{n}}c(x)dx\right|\le C$ for some
$C\in\mathbb{R}^{+}$.

In a similar way, by (\ref{eq:vqdef}) and by (\ref{eq:gij}) we have
\begin{multline}
\delta\Delta_{g}V_{\delta,q}=\\
\frac{\delta}{\delta^{\frac{n-2}{2}}}\left(\frac{1}{\delta^{2}}\Delta_{\text{euc}}v_{q}(x)+\frac{1}{\delta^{2}}[g^{ij}-\delta_{ij}]\partial_{ij}^{2}v_{q}(x)+\delta g(x)\Gamma_{ij}^{k}(x)\frac{1}{\delta}\partial_{k}v_{q}(x)+o(\delta^{2})c(y)\right)\\
=\frac{1}{\delta^{\frac{n}{2}}}\left(-2h_{ij}(0)x_{n}\partial_{ij}^{2}U(y)+O(\delta)c(y)\right)\label{eq:R2}
\end{multline}
Thus, in local chart by (\ref{eq:R1}) and (\ref{eq:R2}) we get
\begin{equation}
\|\Delta_{g}(W_{\delta,q}+\delta V_{\delta,q})\|_{L^{\frac{2n}{n+2}}(M)}=\delta^{n\frac{n+2}{2n}}\frac{1}{\delta^{\frac{n}{2}}}O(\delta)=O(\delta^{2})\label{eq:deltaw+v}
\end{equation}
and we obtain the proof, once we set $\delta=\lambda\varepsilon$. \end{proof}
\begin{rem}
\label{rem:N}We have that the nonlinear operator $N$ (see (\ref{eq:defN}))
is a contraction. By the properties of $i^{*}$ and using the expansion
of $f_{\varepsilon}(W_{\delta,q}+\phi_{1}+\delta V_{\delta,q})$ centered
in$W_{\delta,q}+\phi_{2}+\delta V_{\delta,q}$ we have 
\begin{multline*}
\|N(\phi_{1})-N(\phi_{2})\|_{H}\\
\le\left\Vert \left(f'\left(W_{\delta,q}+\theta\phi_{1}+(1-\theta)\phi_{2}+\delta V_{\delta,q}\right)-f'(W_{\delta,q}+\delta V_{\delta,q})\right)[\phi_{1}-\phi_{2}]\right\Vert _{L^{\frac{2(n-1)}{n}}(\partial M)}
\end{multline*}
and, since $|\phi_{1}-\phi_{2}|^{\frac{2(n-1)}{n}}\in L^{\frac{n}{n-2}}(\partial M)$
and $|f_{\varepsilon}'(\cdot)|^{\frac{2(n-1)}{n}}\in L^{\frac{n}{2}}(\partial M)$,
we have 
\begin{multline*}
\|N(\phi_{1})-N(\phi_{2})\|_{H}\\
\le\left\Vert \left(f'\left(W_{\delta,q}+\theta\phi_{1}+(1-\theta)\phi_{2}+\delta V_{\delta,q}\right)-f'(W_{\delta,q})+\delta V_{\delta,q}\right)\right\Vert _{L^{\frac{2(n-1)}{n-2}}(\partial M)}\|\phi_{1}-\phi_{2}\|_{H}\\
=\beta\|\phi_{1}-\phi_{2}\|_{H}
\end{multline*}
where 
$$\beta=\left\Vert \left(f_{\varepsilon}'\left(W_{\delta,q}+\theta\phi_{1}+(1-\theta)\phi_{2}+\delta V_{\delta,q}\right)-f_{\varepsilon}'(W_{\delta,q}+\delta V_{\delta,q})\right)\right\Vert _{L^{\frac{2(n-1)}{n-2}}(\partial M)}<1,$$
provided $\|\phi_{1}\|_{H}$ and $\|\phi_{2}\|_{H}$ sufficiently
small. 

In the same way we can prove that $\|N(\phi)\|_{H}\le\bar\beta\|\phi\|_{H}$
with $\bar\beta<1$ if $\|\phi\|_{H}$ is sufficiently small.\end{rem}
\begin{prop}
\label{prop:EsistenzaPhi}Let $\delta=\varepsilon\lambda$ For $a,b\in\mathbb{R}$,
$0<a<b$ there exists a positive constant $C=C(a,b)$ such that, for
$\varepsilon$ small, for any $q\in\partial M$, for any $\lambda\in[a,b]$
there exists a unique $\Phi=\Phi_{\varepsilon,\delta,q}\in K_{\delta,q}^{\bot}$
which solves (\ref{eq:P-Kort}) such that 
\[
\|\Phi\|_{H}\le C\varepsilon^{2}
\]
\end{prop}
\begin{proof}
By Remark \ref{rem:N} we have that $N$ is a contraction. Moreover,
by Lemma \ref{prop:L} and by Lemma \ref{lem:R} there exists $C>0$
such that
\[
\left\Vert L^{-1}\left(N(\phi)+R-\Pi^{\bot}\left\{ i^{*}\left(\varepsilon\gamma(W_{\delta,q}+\delta V_{\delta,q}+\phi)\right)\right\} \right)\right\Vert _{H}\le C\left((\beta+\varepsilon)\|\phi\|_{H}+\varepsilon^{2}\right).
\]
In fact, we have 
\begin{align*}
\left\Vert i^{*}\left(\varepsilon\gamma(W_{\varepsilon\lambda,q}+\varepsilon\lambda V_{\varepsilon\lambda,q}+\phi)\right)\right\Vert _{H} & \le\varepsilon\left(\left\Vert W_{\varepsilon\lambda,q}+\varepsilon\lambda V_{\varepsilon\lambda,q}\right\Vert _{L^{\frac{2(n-1)}{n}}}+\left\Vert \phi\right\Vert _{H}\right)\\
 & \le C(\varepsilon^{2}+\varepsilon\left\Vert \phi\right\Vert _{H})
\end{align*}
Notice that, given $C>0$, in Remark \ref{rem:N} it is possible (up
to choose $\|\phi\|_{H}$ sufficiently small) to choose $0<C(\beta+\varepsilon)<1/2$. 

Now, if $\|\phi\|_{H}\le2C\varepsilon^{2}$ then the map
\[
T(\phi):=L^{-1}\left(N(\phi)+R-\Pi^{\bot}\left\{ i^{*}\left(\varepsilon\gamma(W_{\delta,q}+\delta V_{\delta,q}+\phi)\right)\right\} \right)
\]
is a contraction from the ball $\|\phi\|_{H}\le2C\varepsilon^{2}$
in itself, so, by the fixed point Theorem, there exists a unique $\Phi$
with $\|\Phi\|_{H}\le2C\varepsilon^{2}$ solving (\ref{eq:P-Kort}).
The regularity of the map $q\mapsto\Phi$ can be proven via the implicit
function Theorem.
\end{proof}

\section{The reduced functional}\label{tre}
\begin{lem}
\label{lem:JWpiuPhi}Assume $n\ge7$ and $\delta=\lambda\varepsilon$.
It holds 
\[
J_{\varepsilon}(W_{\delta,q}+\delta V_{\delta,q}+\Phi)-J_{\varepsilon}(W_{\delta,q}+\delta V_{\delta,q})=o\left(\varepsilon^{2}\right)
\]
$C^{0}$-uniformly for $q\in\partial M$ and $\lambda$ in a compact
set of $(0,+\infty)$.\end{lem}
\begin{proof}
We know that $\|\Phi\|_{H}=O(\varepsilon^{2})$, so we estimate, for
some $\theta\in(0,1)$
\begin{multline*}
J_{\varepsilon}(W_{\delta,q}+\delta V_{\delta,q}+\Phi)-J_{\varepsilon}(W_{\delta,q}+\delta V_{\delta,q})=J_{\varepsilon}'(W_{\delta,q}+\delta V_{\delta,q})[\Phi]\\
+\frac12 J_{\varepsilon}''(W_{\delta,q}+\delta V_{\delta,q}+\theta\Phi)[\Phi,\Phi]\\
=\int_{M}\left(\nabla_{g}W_{\delta,q}+\delta\nabla_{g}V_{\delta,q}\right)\nabla\Phi+a\left(W_{\delta,q}+\delta V_{\delta,q}\right)\Phi d\mu_{g}\\
+\int_{\partial M}\varepsilon\gamma\left(W_{\delta,q}+\delta V_{\delta,q}\right)\Phi d\sigma-(n-2)\int_{\partial M}\left(\left(W_{\delta,q}+\delta V_{\delta,q}\right)^{+}\right)^{\frac{n}{n-2}}\Phi d\sigma\\
+\frac12 \int_{M}|\nabla\Phi|^{2}+a\Phi^{2}d\mu_{g}+\frac12\int_{\partial M}\varepsilon\gamma\Phi^{2}d\sigma\\
-\frac{n}2\int_{\partial M}\left(\left(W_{\delta,q}+\delta V_{\delta,q}+\theta\Phi\right)^{+}\right)^{\frac{2}{n-2}}\Phi^{2}d\sigma.
\end{multline*}
Immediately we have, by Holder inequality, and setting $\delta=\varepsilon\lambda$,
\[
\int_{M}|\nabla\Phi|^{2}+a\Phi^{2}d\mu_{g}+\int_{\partial M}\varepsilon\gamma\Phi^{2}d\sigma\le C\|\Phi\|_{H}^{2}=o(\varepsilon^{2});
\]
\[
\int_{M}aW_{\delta,q}\Phi d\mu_{g}\le C\|W_{\delta,q}\|_{L^{\frac{2n}{n+2}}(M)}\|\Phi\|_{L^{\frac{2n}{n-2}}(M)}\le C\delta^{2}\|\Phi\|_{H}=o(\varepsilon^{2});
\]
\[
\delta\int_{M}aV_{\delta,q}\Phi d\mu_{g}\le C\delta\|V_{\delta,q}\|_{L^{2}(M)}\|\Phi\|_{L^{2}(M)}\le C\delta^{2}\|\Phi\|_{H}=o(\varepsilon^{2});
\]
\begin{align*}
\int_{\partial M}\varepsilon\gamma\left(W_{\delta,q}+\delta V_{\delta,q}\right)\Phi d\sigma & \le C\varepsilon\|W_{\delta,q}+\delta V_{\delta,q}\|_{L^{\frac{2(n-1)}{n}}(\partial M)}\|\Phi\|_{L^{\frac{2(n-1)}{n-2}}(\partial M)}\\
 & \le\varepsilon C\delta\|\Phi\|_{H}=o(\varepsilon^{2})
\end{align*}
\begin{align*}
\int_{\partial M}\left(\left(W_{\delta,q}+\delta V_{\delta,q}+\theta\Phi\right)^{+}\right)^{\frac{2}{n-2}}\Phi^{2}d\sigma & \le C\|\Phi\|_{H}^{2}\left(\left\Vert W_{\delta,q}+\delta V_{\delta,q}+\theta\Phi\right\Vert _{L^{\frac{2(n-1)}{n-2}}(\partial M)}^{\frac{2}{n-2}}\right)\\
 & \le C\|\Phi\|_{H}^{2}=o(\varepsilon^{2});
\end{align*}
By integration by parts we have 
\begin{multline*}
\int_{M}\left(\nabla_{g}W_{\delta,q}+\delta\nabla_{g}V_{\delta,q}\right)\nabla\Phi d\mu_{g}=-\int_{M}\Delta_{g}\left(W_{\delta,q}+\delta V_{\delta,q}\right)\Phi d\mu_{g}\\
+\int_{\partial M}\left(\frac{\partial}{\partial\nu}W_{\delta,q}+\delta\frac{\partial}{\partial\nu}V_{\delta,q}\right)\Phi d\mu_{g}.
\end{multline*}
and, as in (\ref{eq:deltaw+v}) we get 
\[
\int_{M}\Delta_{g}\left(W_{\delta,q}+\delta V_{\delta,q}\right)\Phi d\mu_{g}\le\|\Delta_{g}(W_{\delta,q}+\delta V_{\delta,q})\|_{L^{\frac{2n}{n+2}}(M)}\|\Phi\|_{H}=O(\delta^{2})\|\Phi\|_{H}=o(\varepsilon^{2})
\]
once we set $\delta=\varepsilon\lambda$. Moreover, by Holder inequality,
\[
\int_{\partial M}\delta\frac{\partial}{\partial\nu}V_{\delta,q}\Phi d\mu_{g}\le\delta\left\Vert \frac{\partial}{\partial\nu}V_{\delta,q}\right\Vert _{L^{\frac{2(n-1)}{n}}(\partial M)}\|\Phi\|_{L^{\frac{2(n-1)}{n-2}}(\partial M)}\le O(\delta)\|\Phi\|_{H}=o(\varepsilon^{2}).
\]
In the end we need to verify that 
\begin{multline*}
\int_{\partial M}\left[(n-2)\left(\left(W_{\delta,q}+\delta V_{\delta,q}\right)^{+}\right)^{\frac{n}{n-2}}-\frac{\partial}{\partial\nu}W_{\delta,q}\right]\Phi d\sigma\\
=\left\Vert (n-2)\left(\left(W_{\delta,q}+\delta V_{\delta,q}\right)^{+}\right)^{\frac{n}{n-2}}-\frac{\partial}{\partial\nu}W_{\delta,q}\right\Vert _{L^{\frac{2(n-1)}{n}}(\partial M)}\|\Phi\|_{L^{\frac{2(n-1)}{n-2}}(\partial M)}\\
=o(1)\|\Phi\|_{H}=o(\varepsilon^{2})
\end{multline*}
In fact, by (\ref{eq:vqdelta}), (\ref{eq:vqdef}) and by taylor expansion
we have 
\begin{multline*}
\int_{\partial M}\left[(n-2)\left(\left(W_{\delta,q}+\delta V_{\delta,q}\right)^{+}\right)^{\frac{n}{n-2}}-\frac{\partial}{\partial\nu}W_{\delta,q}\right]^{\frac{2(n-1)}{n}}d\sigma\\
\le\int_{\partial\mathbb{R}_{+}^{n}}\left[(n-2)\left(\left(U_{\delta}+\delta\left(v_{q}\right)_{\delta}\right)^{+}\right)^{\frac{n}{n-2}}+\frac{\partial}{\partial t}U_{\delta}\right]^{\frac{2(n-1)}{n}}dz+o(1)\\
\le\int_{\partial\mathbb{R}_{+}^{n}}\left[n\left(\left(U_{\delta}+\theta\delta\left(v_{q}\right)_{\delta}\right)^{+}\right)^{\frac{2}{n-2}}\delta\left(v_{q}\right)_{\delta}\right]^{\frac{2(n-1)}{n}}dz+o(1)=o(1),
\end{multline*}
which concludes the proof.\end{proof}
\begin{prop}
\label{lem:expJeps}Assume $n\ge7$ and $\delta=\lambda\varepsilon$.
It holds 
\[
J_{\varepsilon}(W_{\lambda\varepsilon,q}+\lambda\varepsilon V_{\lambda\varepsilon,q})=A+\varepsilon^{2}\left[\lambda B\gamma(q)+\lambda^{2}\varphi(q)\right]+o(\varepsilon^{2}),
\]
$C^{0}$-uniformly for $q\in\partial M$ and $\lambda$ in a compact
set of $(0,+\infty)$, where (see \eqref{new})
\begin{equation}\label{phiq}
\varphi(q)= \frac{1}{2}\int_{\mathbb{R}_{+}^{n}}\Delta v_{q}v_qdzdt-\frac{(n-6)(n-2)\omega_{n-1}I_{n-1}^{n}}{4(n-1)^{2}(n-4)}\|\pi(q)\|^{2}\le0.
\end{equation}
\[
B=\frac{n-2}{n-1}\omega_{n-1}I_{n-1}^{n}>0
\]
and 
\begin{align*}
A & =\frac{1}{2}\int_{\mathbb{R}_{+}^{n}}|\nabla U(z,t)|^{2}dzdt-\frac{(n-2)^{2}}{2(n-1)}\int_{\partial\mathbb{R}_{+}^{n}}U(z,0)^{\frac{2(n-1)}{n-2}}dz\\
 & =\frac{(n-2)(n-3)}{2(n-1)^{2}}\omega_{n-1}I_{n-1}^{n}>0
\end{align*}
\end{prop}
\begin{rem}
Notice that $A$ is the energy level $J_{\infty}(U)=\inf_{u\in H^{1}(\mathbb{R}_{+}^{n})}J_{\infty}(u)$,
where $J_{\infty}$ is the functional associated to the limit equation
(\ref{eq:Udelta}).\end{rem}
\begin{proof}
We expand in $\delta$ the functional 
\begin{align*}
J_{\varepsilon}(W_{\delta,q}+\delta V_{\delta,q})= & \frac{1}{2}\int_{M}|\nabla_{g}W_{\delta,q}+\delta\nabla_{g}V_{\delta,q}|^{2}d\mu_{g}+\frac{1}{2}\int_{M}a\left(W_{\delta,q}+\delta V_{\delta,q}\right)^{2}d\mu_{g}\\
 & +\frac{1}{2}\int_{\partial M}\varepsilon\gamma\left(W_{\delta,q}+\delta V_{\delta,q}\right)^{2}d\sigma\\
 & -\frac{(n-2)^{2}}{2(n-1)}\int_{\partial M}\left[\left(\left(W_{\delta,q}+\delta V_{\delta,q}\right)^{+}\right)^{\frac{2(n-1)}{n-2}}-\left(W_{\delta,q}\right)^{\frac{2(n-1)}{n-2}}\right]d\sigma\\
 & -\frac{(n-2)^{2}}{2(n-1)}\int_{\partial M}\left(W_{\delta,q}\right)^{\frac{2(n-1)}{n-2}}d\sigma=I_{1}+I_{2}+I_{3}+I_{4}+I_{5}.
\end{align*}
For the term $I_{2}$, by Remark \ref{lem:I-a-m} in the appendix,
we have, by change of variables,
\begin{align}
I_{2} & =\frac{1}{2}\delta^{2}\int_{\mathbb{R}_{+}^{n}}\tilde{a}(\delta y)\left(U(y)\chi(\delta y)+\delta v_{q}(y)\chi(\delta y)\right)^{2}|g(\delta y)|^{1/2}dy\nonumber \\
 & =\frac{1}{2}\delta^{2}a(q)\int_{\mathbb{R}_{+}^{n}}U(y)^{2}dy+o(\delta^{2})\nonumber \\
 & =\delta^{2}a(q)\frac{n-2}{(n-1)(n-4)}\omega_{n-1}I_{n-1}^{n}+o(\delta^{2})\label{eq:sviluppo1}
\end{align}
in fact by Remark \ref{lem:I-a-m} we have
\begin{align*}
\int_{\mathbb{R}_{+}^{n}}U(y)^{2}dy & =\frac{1}{n-4}\omega_{n-1}I_{n-2}^{n-2}=\frac{2(n-2)}{(n-4)(n-1)}\omega_{n-1}I_{n-1}^{n}
\end{align*}
For the term $I_{3}$, recalling that $y=(z,t)$ with $z\in\mathbb{R}^{n-1}$,
$t\ge0$, we have, by Remark \ref{lem:I-a-m}, 
\begin{align}
I_{3} & =\frac{\varepsilon\delta}{2}\int_{\mathbb{R}^{n-1}}\tilde{\gamma}(0,\delta z)\left(U(0,z)\chi(0,\delta z)+\delta v_{q}(0,z)\chi(0,\delta z)\right)^{2}|g(0,\delta z)|^{1/2}dz\nonumber \\
 & =\frac{\varepsilon\delta}{2}\gamma(q)\int_{\mathbb{R}^{n-1}}U(0,z)^{2}dz+o(\varepsilon\delta)=\frac{\varepsilon\delta}{2}\gamma(q)\int_{0}^{\infty}\frac{1}{\left[1+|z|^{2}\right]^{n-2}}dz\nonumber \\
 & =\varepsilon\delta\frac{\gamma(q)}{2}\omega_{n-1}I_{n-2}^{n-2}=\varepsilon\delta\gamma(q)\frac{n-2}{n-1}\omega_{n-1}I_{n-1}^{n}\label{eq:sviluppo2}
\end{align}
For the term $I_{5}$, by (\ref{eq:|g|}) we have 
\begin{align*}
I_{5} & =-\frac{(n-2)^{2}}{2(n-1)}\int_{\mathbb{R}^{n-1}}\left(U(0,z)\chi(0,\delta z)\right)^{\frac{2(n-1)}{n-2}}|g(0,\delta z)|^{1/2}dz\\
 & =-\frac{(n-2)^{2}}{2(n-1)}\int_{\mathbb{R}^{n-1}}U(0,z)^{\frac{2(n-1)}{n-2}}\left(1-\frac{\delta^{2}}{6}\bar{R}_{ij}(q)z_{i}z_{j}\right)dz+o(\delta^{2});
\end{align*}
by Remark \ref{lem:I-a-m} it holds 
\[
\int_{\mathbb{R}^{n-1}}U(0,z)^{\frac{2(n-1)}{n-2}}=\omega_{n-1}I_{n-1}^{n-2}
\]
and, by symmetry reasons, 
\begin{align*}
\bar{R}_{ij}(q)\int_{\mathbb{R}^{n-1}}U(0,z)^{\frac{2(n-1)}{n-2}}z_{i}z_{j}dz & =\sum_{i=1}^{n-1}\bar{R}_{ii}(q)\int_{\mathbb{R}^{n-1}}U(0,z)^{\frac{2(n-1)}{n-2}}z_{i}^{2}dz\\
= & \frac{\bar{R}_{ii}(q)}{n-1}\int_{\mathbb{R}^{n-1}}\frac{|z|^{2}dz}{(1+|z|^{2})^{n-1}}=\frac{\bar{R}_{ii}(q)}{n-1}\omega_{n-1}I_{n-1}^{n}.
\end{align*}
Thus, since $I_{n-1}^{n-2}=\frac{n-3}{n-1}I_{n-1}^{n}$ by Remark
\ref{lem:I-a-m}, 
\begin{align}
I_{5} & =-\frac{(n-2)^{2}}{2(n-1)}\omega_{n-1}\left(I_{n-1}^{n-2}-\frac{\delta^{2}}{6(n-1)}\bar{R}_{ii}(q)\omega_{n-1}I_{n-1}^{n}\right)\nonumber \\
 & =-\frac{(n-2)^{2}(n-3)}{2(n-1)^{2}}\omega_{n-1}I_{n-1}^{n}+\delta^{2}\frac{(n-2)^{2}}{12(n-1)^{2}}\bar{R}_{ii}(q)\omega_{n-1}I_{n-1}^{n}.\label{eq:sviluppo3}
\end{align}
For the term $I_{1}$ we write
\[
I_{1}=\frac{1}{2}\int_{M}|\nabla_{g}W_{\delta,q}|^{2}+\frac{1}{2}\int_{M}2\delta\nabla W_{\delta,q}\nabla V_{\delta,q}+\delta^{2}|\nabla_{g}V_{\delta,q}|^{2}d\mu_{g}=I_{1}'+I_{1}''+I_{1}'''
\]
and we proceed by estimating each term separately. By (\ref{eq:|g|}),
(\ref{eq:gin}), (\ref{eq:gij}), we have (here $a,b=1,\dots,n$ and
$i,j,m,l=1,\dots,n-1$) 
\begin{multline*}
I_{1}'=\frac{1}{2}\int_{\mathbb{R}_{+}^{n}}g^{ab}(\delta y)\frac{\partial}{\partial y_{a}}(U(y)\chi(\delta y))\frac{\partial}{\partial y_{b}}(U(y)\chi(\delta y))|g(\delta y)|^{1/2}dy\\
=\int_{\mathbb{R}_{+}^{n}}\left[\frac{|\nabla U|^{2}}{2}+\left(\delta h_{ij}t-\frac{\delta^{2}}{6}\bar{R}_{ikjl}z_{k}z_{l}+\delta^{2}\frac{\partial h_{ij}}{\partial z_{k}}tz_{k}+\frac{\delta}{2}^{2}\left[R_{injn}+3h_{ik}h_{kj}\right]t^{2}\right)\frac{\partial U}{\partial z_{i}}\frac{\partial U}{\partial z_{j}}\right]\\
\times\left(1-\frac{\delta^{2}}{2}\left[\|\pi\|^{2}+\ric(0)\right]t^{2}-\frac{\delta^{2}}{6}\bar{R}_{lm}(0)z_{l}z_{m}\right)dzdt+o(\delta^{2}).
\end{multline*}
Since $\frac{\partial U}{\partial z_{i}}=(2-n)\frac{z_{i}}{\left[(1+t)^{2}+|z|^{2}\right]^{\frac{n}{2}}}$,
by symmetry reasons and since $h_{ii}\equiv0$ we have that 
\[
h_{ij}(q)\int_{\mathbb{R}_{+}^{n}}t\frac{\partial U}{\partial z_{i}}\frac{\partial U}{\partial z_{j}}dzdt=h_{ii}(q)\int_{\mathbb{R}_{+}^{n}}\frac{tz_{i}z_{i}dzdt}{\left[(1+t)^{2}+|z|^{2}\right]^{\frac{n}{2}}}=0
\]
\[
\frac{\partial h_{ij}}{\partial z_{k}}(q)\int_{\mathbb{R}_{+}^{n}}tz_{k}\frac{\partial U}{\partial z_{i}}\frac{\partial U}{\partial z_{j}}dzdt=(2-n)\frac{\partial h_{ij}}{\partial z_{k}}(q)\int_{\mathbb{R}_{+}^{n}}\frac{tz_{k}z_{i}z_{j}dzdt}{\left[(1+t)^{2}+|z|^{2}\right]^{\frac{n}{2}}}=0;
\]
in a similar way, using the symmetries of the curvature tensor one
can check that 
\begin{align*}
\bar{R}_{ikjl}(q)\int_{\mathbb{R}_{+}^{n}}z_{k}z_{l}\frac{\partial U}{\partial z_{i}}\frac{\partial U}{\partial z_{j}}dzdt & =\bar{R}_{ikjl}(q)\int_{\mathbb{R}_{+}^{n}}\frac{z_{i}z_{j}z_{k}z_{l}dzdt}{\left[(1+t)^{2}+|z|^{2}\right]^{\frac{n}{2}}}\\
 & =\frac{\alpha}{3}\left(R_{ikik}(q)+R_{ikki}(q)+R_{iijj}(q)\right)=0
\end{align*}
where $\alpha=\int_{\mathbb{R}_{+}^{n}}\frac{z_{1}^{4}dzdt}{\left[(1+t)^{2}+|z|^{2}\right]^{\frac{n}{2}}}$.
Thus, using again symmetry 
\begin{align*}
I_{1}'= & \int_{\mathbb{R}_{+}^{n}}\left[\frac{|\nabla U|^{2}}{2}+\left(\frac{\delta}{2}^{2}\left[R_{injn}+3h_{ik}h_{kj}\right]t^{2}\right)\frac{\partial U}{\partial z_{i}}\frac{\partial U}{\partial z_{j}}\right]\\
 & \times\left(1-\frac{\delta^{2}}{2}\left[\|\pi\|^{2}+\ric(0)\right]t^{2}-\frac{\delta^{2}}{6}\bar{R}_{lm}(0)z_{l}z_{m}\right)dzdt+o(\delta^{2})\\
= & \frac{(n-2)^{2}}{2}\int_{\mathbb{R}_{+}^{n}}\frac{dzdt}{\left[(1+t)^{2}+|z|^{2}\right]^{n-1}}\\
 & +\frac{\delta}{2}^{2}\frac{(n-2)^{2}}{n-1}\left[\ric(q)+3\|\pi(q)\|^{2}\right]\int_{\mathbb{R}_{+}^{n}}\frac{|z|^{2}t^{2}dzdt}{\left[(1+t)^{2}+|z|^{2}\right]^{n}}\\
 & -\frac{\delta^{2}(n-2)^{2}}{4}\left[\|\pi(q)\|^{2}+\ric(q)\right]\int_{\mathbb{R}_{+}^{n}}\frac{t^{2}dzdt}{\left[(1+t)^{2}+|z|^{2}\right]^{n-1}}\\
 & -\frac{\delta^{2}}{12}\frac{(n-2)^{2}}{n-1}\bar{R}_{ll}(q)\int_{\mathbb{R}_{+}^{n}}\frac{|z|^{2}dzdt}{\left[(1+t)^{2}+|z|^{2}\right]^{n-1}}+o(\delta^{2}).
\end{align*}
Thus, by Remark \ref{lem:I-a-m}, 
\begin{align}
I_{1}'= & \frac{(n-2)\omega_{n-1}I_{n-1}^{n-2}}{2}+\delta^{2}\frac{(n-2)\omega_{n-1}I_{n}^{n}}{(n-1)(n-3)(n-4)}\left[\ric(q)+3\|\pi(q)\|^{2}\right]\nonumber \\
 & -\delta^{2}\frac{(n-2)\omega_{n-1}I_{n-1}^{n-2}}{2(n-3)(n-4)}\left[\ric(q)+\|\pi(q)\|^{2}\right]\nonumber \\
 & -\delta^{2}\frac{(n-2)^{2}\omega_{n-1}I_{n-1}^{n}}{12(n-1)(n-4)}\bar{R}_{ll}(q)+o(\delta^{2})\nonumber \\
= & \frac{(n-2)(n-3)}{2(n-1)}\omega_{n-1}I_{n-1}^{n}+\delta^{2}\frac{(n-2)}{2(n-1)^{2}(n-4)}\omega_{n-1}I_{n-1}^{n}\left[\ric(q)+3\|\pi(q)\|^{2}\right]\nonumber \\
 & -\delta^{2}\frac{(n-2)}{2(n-1)(n-4)}\omega_{n-1}I_{n-1}^{n}\left[\ric(q)+\|\pi(q)\|^{2}\right]\nonumber \\
 & -\delta^{2}\frac{(n-2)^{2}}{12(n-1)(n-4)}\bar{R}_{ll}(q)\omega_{n-1}I_{n-1}^{n}+o(\delta^{2})\label{eq:sviluppo4}
\end{align}
For the term $I_{1}''$, by (\ref{eq:|g|}), (\ref{eq:gij}), (\ref{eq:gin})
and by definition of $V_{\delta,q}$ and $v_{q}$ we have 
\begin{align}
I_{1}'' & =\delta\int_{M}\nabla W_{\delta,q}\nabla V_{\delta,q}d\mu_{g}=\delta\int_{\mathbb{R}_{+}^{n}}g^{\alpha\beta}(\delta y)\frac{\partial}{\partial y_{\alpha}}(U(y)\chi(\delta y))\frac{\partial}{\partial y_{\beta}}(v_{q}(y)\chi(\delta y))|g(\delta y)|^{1/2}dy\nonumber \\
 & =\delta\int_{\mathbb{R}_{+}^{n}}\nabla U\nabla v_{q}dy+\delta^{2}2h_{ij}(q)\int_{\mathbb{R}_{+}^{n}}t\frac{\partial U}{\partial y_{i}}\frac{\partial v_{q}}{\partial y_{j}}dy+o(\delta^{2})\nonumber \\
 & =\delta^{2}2h_{ij}(q)\int_{\mathbb{R}_{+}^{n}}t\frac{\partial U}{\partial z_{i}}\frac{\partial v_{q}}{\partial z_{j}}dy+o(\delta^{2})\label{eq:sviluppo5}
\end{align}
in fact 
\begin{align*}
\int_{\mathbb{R}_{+}^{n}}\nabla U\nabla v_{q}dy & =-\int_{\mathbb{R}_{+}^{n}}U\Delta vdy+\int_{\partial\mathbb{R}_{+}^{n}}U(0,z)\frac{\partial v_{q}}{\partial t}dz\\
 & =2h_{ij}\int_{\mathbb{R}_{+}^{n}}Ut\frac{\partial^{2}U}{\partial z_{i}\partial z_{j}}-n\int_{\partial\mathbb{R}_{+}^{n}}U(0,z)\left(U(0,z)^{\frac{2}{n-2}}v_{q}\right)dz=0
\end{align*}
since the first term is zero by symmetry and using that $h_{ii}=0$,
and the second term is zero by (\ref{eq:vqdef}) and (\ref{eq:Uvq}).

For the term $I_{1}'''$, immediately we have 
\begin{equation}
I_{1}'''=\frac{\delta^{2}}{2}\int_{M}|\nabla V_{\delta,q}|^{2}d\mu_{g}=\frac{\delta^{2}}{2}\int_{\mathbb{R}_{+}^{n}}|\nabla v_{q}|^{2}dzdt+o(\delta^{2}),\label{eq:sviluppo6}
\end{equation}
so 
\begin{equation}
I_{1}''+I_{1}'''=\delta^{2}2h_{ij}(q)\int_{\mathbb{R}_{+}^{n}}t\frac{\partial U}{\partial z_{i}}\frac{\partial v_{q}}{\partial z_{j}}dzdt+\frac{\delta^{2}}{2}\int_{\mathbb{R}_{+}^{n}}|\nabla v_{q}|^{2}dzdt+o(\delta^{2})\label{eq:sviluppo7}
\end{equation}
For the term $I_{4}$, by (\ref{eq:Uvq}) and (\ref{eq:|g|}), and
recalling that $y=(z,t)$ we have 
\begin{align}
I_{4}= & -\frac{(n-2)^{2}}{2(n-1)}\int_{\partial\mathbb{R}_{+}^{n}}\left[\left(\left(U+\delta v_{q}\right)^{+}\right)^{\frac{2(n-1)}{n-2}}-U^{\frac{2(n-1)}{n-2}}\right]|g(0,\delta z)|^{\frac{1}{2}}dz+o(\delta^{2})\nonumber \\
= & -\delta(n-2)\int_{\partial\mathbb{R}_{+}^{n}}U^{\frac{n}{n-2}}v_{q}dz-\delta^{2}\frac n2\int_{\partial\mathbb{R}_{+}^{n}}\left(\left(U+\delta v_{q}\right)^{+}\right)^{\frac{2}{n-2}}v_{q}^{2}dz+o(\delta^{2})\nonumber \\
= & -\delta^{2}\frac n2\int_{\partial\mathbb{R}_{+}^{n}}U^{\frac{2}{n-2}}v_{q}^{2}dz+o(\delta^{2}).\label{eq:sviluppo8}
\end{align}
At this point we observe that 
\begin{equation}
2h_{ij}(q)\int_{\mathbb{R}_{+}^{n}}t\frac{\partial U}{\partial z_{i}}\frac{\partial v_{q}}{\partial z_{j}}dzdt-n\int_{\mathbb{R}^{n-1}}U^{\frac{2}{n-2}}v_{q}^{2}dz=-\int_{\mathbb{R}_{+}^{n}}|\nabla v_{q}|^{2}dzdt\label{eq:vqriduzione}
\end{equation}
in fact, by (\ref{eq:vqdef}) we get
\begin{align}
2h_{ij}(q)\int_{\mathbb{R}_{+}^{n}}t\frac{\partial U}{\partial z_{i}}\frac{\partial v_{q}}{\partial z_{j}}dzdt & =-2h_{ij}(q)\int_{\mathbb{R}_{+}^{n}}t\frac{\partial^{2}U}{\partial z_{j}\partial z_{i}}v_{q}dzdt=\int_{\mathbb{R}_{+}^{n}}\left(\Delta v_{q}\right)v_{q}dzdt\nonumber \\
 & =-\int_{\mathbb{R}_{+}^{n}}|\nabla v_{q}|^{2}dzdt+\int_{\partial\mathbb{R}_{+}^{n}}v_{q}\frac{\partial v_{q}}{\partial\nu}dz\nonumber \\
 & =-\int_{\mathbb{R}_{+}^{n}}|\nabla v_{q}|^{2}dzdt+n\int_{\partial\mathbb{R}_{+}^{n}}U^{\frac{2}{n-2}}v_{q}^{2}dz.\label{eq:sviluppo9}
\end{align}
Hence by (\ref{eq:sviluppo7}), (\ref{eq:sviluppo8}), (\ref{eq:sviluppo9}) and \eqref{eq:vqdef}
it holds

\begin{equation}\begin{aligned}
I_{1}''+I_{1}'''+I_{4}&= \delta^{2}  \left(-\frac12\int_{\mathbb{R}_{+}^{n}}|\nabla v_{q}|^{2}dzdt+\frac n2\int_{\partial\mathbb{R}_{+}^{n}}U^{\frac{2}{n-2}}v_{q}^{2}dz\right) +o(\delta^{2})\\
&= \frac12\delta^{2} \int_{\mathbb{R}_{+}^{n}}\Delta v_{q}v_qdzdt+o(\delta^{2})\end{aligned}\label{eq:sviluppo10}
\end{equation}
In light of (\ref{eq:sviluppo1}), (\ref{eq:sviluppo2}), (\ref{eq:sviluppo3}),
(\ref{eq:sviluppo4}), (\ref{eq:sviluppo10}), finally we get 
\begin{multline*}
J_{\varepsilon}(W_{\delta,q}+\delta V_{\delta,q})=\frac{(n-2)(n-3)}{2(n-1)^{2}}\omega_{n-1}I_{n-1}^{n}+\varepsilon\delta\gamma(q)\frac{n-2}{n-1}\omega_{n-1}I_{n-1}^{n}\\
+\frac12\delta^{2} \int_{\mathbb{R}_{+}^{n}}\Delta v_{q}v_qdzdt+\delta^{2}a(q)\frac{n-2}{(n-1)(n-4)}\omega_{n-1}I_{n-1}^{n}\\
-\delta^{2}\frac{(n-2)^{2}}{4(n-1)^{2}(n-4)}\omega_{n-1}I_{n-1}^{n}\left[2\ric(q)+2\frac{n-4}{n-2}\|\pi(q)\|^{2}+\bar{R}_{ii}(q)\right]+o(\delta^{2})
\end{multline*}
Now, we choose $\delta=\lambda\varepsilon$, where $\lambda\in[\alpha,\beta]$,
with for some positive $\alpha,\beta.$ Recalling that $a=\frac{n-2}{4(n-1)}R_{g}$
and that $R_{g}(q)=2\ric(q)+\bar{R}_{ii}(q)+\|\pi(q)\|^{2}$ (see
\cite{E92}) we have the proof.
\end{proof}

\section{Proof of Theorem \ref{thm:main}}\label{quattro}
\begin{lem}
\label{lem:punticritici}If $(\bar{\lambda},\bar{q})\in(0,+\infty)\times\partial M$
is a critical point for the reduced functional 
\[
I_{\varepsilon}(\lambda,q):=J_{\varepsilon}(W_{\varepsilon\lambda,q}+\varepsilon\lambda V_{\varepsilon\lambda,q}+\Phi_{\varepsilon\lambda,q})
\]
then the function $W_{\varepsilon\lambda,q}+\varepsilon\lambda V_{\varepsilon\lambda,q}+\Phi$
is a solution of (\ref{eq:P}). Here $\Phi_{\varepsilon\lambda,q}=\Phi_{\varepsilon,\lambda\varepsilon,q}$
is defined in Proposition \ref{prop:EsistenzaPhi}.\end{lem}
\begin{proof}
Set $q=q(y)=\psi_{\bar{q}}^{\partial}(y)$. Since $(\bar{\lambda},\bar{q})$
is a critical point for the $I_{\varepsilon}(\lambda,q)$ we have,
for $h=1,\dots,n-1$, 

\begin{align*}
0= & \left.\frac{\partial}{\partial y_{h}}I_{\varepsilon}(\bar{\lambda},q(y))\right|_{y=0}\\
= & \langle\!\langle W_{\varepsilon\bar{\lambda},q(y)}+\varepsilon\bar{\lambda}V_{\varepsilon\bar{\lambda},q(y)}+\Phi_{\varepsilon\bar{\lambda},q(y)}-i^{*}\left(f(W_{\varepsilon\bar{\lambda},q(y)}+\varepsilon\bar{\lambda}V_{\varepsilon\bar{\lambda},q(y)}+\Phi_{\varepsilon\bar{\lambda},q(y)})\right)\\
 & -\varepsilon\gamma(W_{\varepsilon\bar{\lambda},q(y)}+\varepsilon\bar{\lambda}V_{\varepsilon\bar{\lambda},q(y)}+\Phi_{\varepsilon\bar{\lambda},q(y)}),\left.\frac{\partial}{\partial y_{h}}(W_{\varepsilon\bar{\lambda},q(y)}+\varepsilon\bar{\lambda}V_{\varepsilon\bar{\lambda},q(y)}+\Phi_{\varepsilon\bar{\lambda},q(y)})\rangle\!\rangle_{H}\right|_{y=0}\\
= & \sum_{i=1}^{n}c_{\varepsilon}^{i}\left.\langle\!\langle Z_{\varepsilon\bar{\lambda},q(y)}^{i},\frac{\partial}{\partial y_{h}}(W_{\varepsilon\bar{\lambda},q(y)}+\varepsilon\bar{\lambda}V_{\varepsilon\bar{\lambda},q(y)}+\Phi_{\varepsilon\bar{\lambda},q(y)})\rangle\!\rangle_{H}\right|_{y=0}\\
= & \sum_{i=1}^{n}c_{\varepsilon}^{i}\left.\langle\!\langle Z_{\varepsilon\bar{\lambda},q(y)}^{i},\frac{\partial}{\partial y_{h}}W_{\varepsilon\bar{\lambda},q(y)}\rangle\!\rangle_{H}\right|_{y=0}+\varepsilon\bar{\lambda}\sum_{i=1}^{n}c_{\varepsilon}^{i}\left.\langle\!\langle Z_{\varepsilon\bar{\lambda},q(y)}^{i},\frac{\partial}{\partial y_{h}}V_{\varepsilon\bar{\lambda},q(y)}\rangle\!\rangle_{H}\right|_{y=0}\\
 & \sum_{i=1}^{n}c_{\varepsilon}^{i}\left.\langle\!\langle\frac{\partial}{\partial y_{h}}Z_{\varepsilon\bar{\lambda},q(y)}^{i},\Phi_{\varepsilon\bar{\lambda},q(y)}\rangle\!\rangle_{H}\right|_{y=0}
\end{align*}
using that $\Phi_{\varepsilon\bar{\lambda},q(y)}$ is a solution of
(\ref{eq:P-Kort}) and that 
\[
\langle\!\langle Z_{\varepsilon\bar{\lambda},q(y)}^{i},\frac{\partial}{\partial y_{h}}\Phi_{\varepsilon\bar{\lambda},q(y)}\rangle\!\rangle_{H}=\langle\!\langle\frac{\partial}{\partial y_{h}}Z_{\varepsilon\bar{\lambda},q(y)}^{i},\Phi_{\varepsilon\bar{\lambda},q(y)}\rangle\!\rangle_{H}
\]
since $\Phi_{\varepsilon\bar{\lambda},q(y)}\in K_{\varepsilon\bar{\lambda},q(y)}^{\bot}$
for any $y$.

Arguing as in Lemma 6.1 and Lemma 6.2 of \cite{MP09} we have 
\begin{eqnarray*}
\left\Vert \frac{\partial}{\partial y_{h}}Z_{\varepsilon\bar{\lambda},q(y)}^{i}\right\Vert _{H}=O\left(\frac{1}{\varepsilon}\right) &  & \left\Vert \frac{\partial}{\partial y_{h}}W_{\varepsilon\bar{\lambda},q(y)}\right\Vert _{H}=O\left(\frac{1}{\varepsilon}\right)\\
\left\Vert \frac{\partial}{\partial y_{h}}V_{\varepsilon\bar{\lambda},q(y)}\right\Vert _{H}=O\left(\frac{1}{\varepsilon}\right)
\end{eqnarray*}
so we get
\begin{align*}
\langle\!\langle Z_{\varepsilon\bar{\lambda},q(y)}^{i},\frac{\partial}{\partial y_{h}}W_{\varepsilon\bar{\lambda},q(y)})\rangle\!\rangle_{H}= & \frac{1}{\lambda\varepsilon}\langle\!\langle Z_{\varepsilon\bar{\lambda},q(y)}^{i},Z_{\varepsilon\bar{\lambda},q(y)}^{h})\rangle\!\rangle_{H}+o(1)=\frac{\delta_{ih}}{\lambda\varepsilon}+o(1)\\
\langle\!\langle Z_{\varepsilon\bar{\lambda},q(y)}^{i},\frac{\partial}{\partial y_{h}}V_{\varepsilon\bar{\lambda},q(y)}\rangle\!\rangle_{H} & \le\left\Vert Z_{\varepsilon\bar{\lambda},q(y)}^{i}\right\Vert _{H}\left\Vert \frac{\partial}{\partial y_{h}}V_{\varepsilon\bar{\lambda},q(y)}\right\Vert _{H}=O\left(\frac{1}{\varepsilon}\right)\\
\langle\!\langle\frac{\partial}{\partial y_{h}}Z_{\varepsilon\bar{\lambda},q(y)}^{i},\Phi_{\varepsilon\bar{\lambda},q(y)}\rangle\!\rangle_{H} & \le\left\Vert \frac{\partial}{\partial y_{h}}Z_{\varepsilon\bar{\lambda},q(y)}^{i}\right\Vert _{H}\left\Vert \Phi_{\varepsilon\bar{\lambda},q(y)}\right\Vert _{H}=o(1).
\end{align*}
We conclude that
\[
0=\frac{1}{\lambda\varepsilon}\sum_{i=1}^{n}c_{\varepsilon}^{i}\left(\delta_{ih}+O(1)\right)
\]
and so $c_{\varepsilon}^{i}=0$ for $i=1,\dots,n$. 

Analogously we proceed for $\left.\frac{\partial}{\partial\lambda}I_{\varepsilon}(\lambda,\bar{q})\right|_{\lambda=\bar{\lambda}}$. 
\end{proof}
For the sake of completeness, we recall the definition of $C^{0}$-stable
critical point before proving Theorem \ref{thm:main}.
\begin{defn}
Let $f:\mathbb{R}^{n}\rightarrow\mathbb{R}$ be a $C^{1}$ function
and let $K=\left\{ \xi\in\mathbb{R}^{n}\ :\ \nabla f(\xi)=0\right\} $.
We say that $\xi_{0}\in\mathbb{R}^{n}$ is a $C^{0}$-stable critical
point if $\xi_{0}\in K$ and there exist $\Omega$ neighborhood of
$\xi_{0}$ with $\partial\Omega\cap K=\emptyset$ and a $\eta>0$
such that for any $g:\mathbb{R}^{n}\rightarrow\mathbb{R}$ of class
$C^{1}$ with $\|g-f\|_{C^{0}(\bar{\Omega})}\le\eta$ we have a critical
point of $g$ near $\Omega$.\end{defn}
\begin{proof}[Proof of Theorem \ref{thm:main}]
 Let us call 
\begin{equation}\label{ridotta}
G(\lambda,q)=\lambda B\gamma(q)+\lambda^{2}\varphi(q).
\end{equation}
If we find a $C^{0}$-stable critical point for $G(\lambda,q)$ then
we find a critical point for $I_{\varepsilon}(\lambda,q):=J_{\varepsilon}(W_{\lambda\varepsilon,q}+\lambda\varepsilon V_{\lambda\varepsilon,q}+\Phi)$
for $\varepsilon$ small enough (see Lemma \ref{lem:JWpiuPhi} and
Proposition \ref{lem:expJeps}), hence a solution for Problem (\ref{eq:P}),
by Lemma \ref{lem:punticritici}. 

Since we assumed the trace-free second fundamental form to be nonzero
everywhere, we have $\|\pi\|^{2}>0$,  so $\varphi(q)<0$. 

Also, we assumed $\gamma(q)$ to be strictly positive on $\partial M$,
so there exists $(\lambda_{0},q_{0})$ maximum point of $G(\lambda,q)$
with $\lambda_{0}>0$. Moreover, $(\lambda_{0},q_{0})$ is a $C^{0}$-stable
critical point of $G(\lambda,q)$. Then, for any sufficiently small
$\varepsilon>0$ there exists $(\lambda_{\varepsilon},q_{\varepsilon})$
critical point for $I_{\varepsilon}(\lambda,q)$ and we completed
the proof of our main result, in fact we found a sequence $\lambda_{\varepsilon}$
bounded away from zero, a sequence of points $q_{\varepsilon}\in\partial M$
and a sequence of positive functions 
\[
u_{\varepsilon}=W_{\lambda_{\varepsilon}\varepsilon,q_{\varepsilon}}+\lambda_{\varepsilon}\varepsilon V_{\lambda_{\varepsilon}\varepsilon,q_{\varepsilon}}+\Phi
\]
which are solution for (\ref{eq:P}) with $q_{\varepsilon}\rightarrow q_{0}$. \end{proof}
\begin{rem}
We give another example of function $\gamma(q)$ such that problem
(\ref{eq:P}) admits a positive solution. Let $q_{0}\in\partial M$
be a maximum point for $\varphi$. This point exists since $\partial M$
is compact. Now choose $\gamma\in C^{2}(\partial M)$ such that $\gamma$
has a positive local maximum in $q_{0}$. Then the pair $(\lambda_{0},q_{0})=\left(-\frac{B\gamma(q_{0})}{2\varphi(q_{0})},q_{0}\right)$
is a $C^{0}$-stable critical point for $G(\lambda,q)$.
\end{rem}
In fact, we have 
\[
\nabla_{\lambda,q}G=(B\gamma(q)+2\lambda\varphi(q),\lambda B\nabla_{q}\gamma(q)+\lambda^{2}\nabla_{q}\varphi(q))
\]
which vanishes for $(\lambda_{0},q_{0})=\left(-\frac{B\gamma(q_{0})}{2\varphi(q_{0})},q_{0}\right)$.
Moreover the Hessian matrix is
\[
G_{\lambda,q}^{''}\left(-\frac{B\gamma(q_{0})}{2\varphi(q_{0})},q_{0}\right)=\left(\begin{array}{cc}
2\varphi(q_{0}) & 0\\
0 & -\frac{B^{2}\gamma(q_{0})}{2\varphi(q_{0})}\gamma''_{q}(q_{0})+\frac{B^{2}\gamma^{2}(q_{0})}{\varphi^{2}(q_{0})}\varphi''_{q}(q_{0})
\end{array}\right)
\]
which is negative definite. Thus $(\lambda_{0},q_{0})=\left(-\frac{B\gamma(q_{0})}{2\varphi(q_{0})},q_{0}\right)$
is a maximum $C^{0}$-stable point for $G(\lambda,q)$.

\section{Appendix}
\begin{proof}[Proof of Lemma \ref{prop:L}]
 We argue by contradiction. We suppose that there exist two sequence
of real numbers $\varepsilon_{m}\rightarrow0,\lambda_{m}\in[a,b]$
a sequence of points $q_{m}\in\partial M$ and a sequence of functions
$\phi_{\varepsilon_{m}\lambda_{m},q_{m}}\in K_{\varepsilon_{m}\lambda_{m},q_{m}}^{\bot}$
such that 
\[
\|\phi_{\varepsilon_{m}\lambda_{m},q_{m}}\|_{H}=1\text{ and }\|L_{\varepsilon_{m}\lambda_{m},q_{m}}(\phi_{\varepsilon_{m}\lambda_{m},q_{m}})\|_{H}\rightarrow0\text{ as }m\rightarrow+\infty.
\]
For the sake of simplicity, we set $\delta_{m}=\varepsilon_{m}\lambda_{m}$
and we define
\[
\tilde{\phi}_{m}:=\delta_{m}^{\frac{n-2}{2}}\phi_{\delta_{m},q_{m}}(\psi_{q_{m}}^{\partial}(\delta_{m}y))\chi(\delta_{m}y)\text{ for }y=(z,t)\in\mathbb{R}_{+}^{n},\text{ with }z\in\mathbb{R}^{n-1}\text{ and }t\ge0
\]
Since $\|\phi_{\varepsilon_{m}\lambda_{m},q_{m}}\|_{H}=1$, by change
of variables we easily get that $\left\{ \tilde{\phi}_{m}\right\} _{m}$
is bounded in $D^{1,2}(\mathbb{R}_{+}^{n})$ (but not in $H^{1}(\mathbb{R}_{+}^{n})$).
Thus there exists $\tilde{\phi}\in D^{1,2}(\mathbb{R}_{+}^{n})$ such
that $\tilde{\phi}_{m}\rightharpoonup\tilde{\phi}$ weakly in $D^{1,2}(\mathbb{R}_{+}^{n})$,
in $L^{\frac{2n}{n-2}}(\mathbb{R}_{+}^{n})$ and in $L^{\frac{2(n-1)}{n-2}}(\partial\mathbb{R}_{+}^{n})$,
strongly in $L_{\text{loc}}^{s}(\partial\mathbb{R}_{+}^{n})$ for
$s\le\frac{2(n-1)}{n-2}$ and almost everywhere.

Since $\phi_{\delta_{m},q_{m}}\in K_{\delta_{m},q_{m}}^{\bot}$, and
taking in account (\ref{eq:linearizzato}) we get, for $i=1,\dots,n$,
\begin{equation}
o(1)=\int_{\mathbb{R}_{+}^{n}}\nabla\tilde{\phi}\nabla j_{i}dzdt=n\int_{\partial\mathbb{R}_{+}^{n}}U^{\frac{2}{n-2}}(z,0)j_{i}(z,0)\tilde{\phi}(z,0)dz.\label{eq:L3}
\end{equation}
 Indeed, by change of variables we have
\begin{align*}
0= & \left\langle \left\langle \phi_{\delta_{m},q_{m}},Z_{\delta_{m},q_{m}}^{i}\right\rangle \right\rangle _{H}=\int_{M}\left(\nabla_{g}\phi_{\delta_{m},q_{m}}\nabla_{g}Z_{\delta_{m},q_{m}}^{i}+a\phi_{\delta_{m},q_{m}}Z_{\delta_{m},q_{m}}^{i}\right)d\mu_{g}\\
= & \int_{\mathbb{R}_{+}^{n}}\delta^{\frac{n-2}{2}}\frac{\partial}{\partial\eta_{\alpha}}j_{i}(y)\frac{\partial}{\partial\eta_{\alpha}}\phi_{\delta_{m},q_{m}}(\psi_{q_{m}}^{\partial}(\delta_{m}y))dy\\
 & +\int_{\mathbb{R}_{+}^{n}}\delta^{\frac{n+2}{2}}a(\psi_{q_{m}}^{\partial}(\delta y))j_{i}(y)\phi_{\delta_{m},q_{m}}(\psi_{q_{m}}^{\partial}(\delta_{m}y))dy+o(1)\\
= & \int_{\mathbb{R}_{+}^{n}}\nabla j_{i}(y)\nabla\tilde{\phi}_{m}(y)+\delta^{2}a(q_{m})j_{i}(y)\tilde{\phi}_{m}(y)d\eta+o(1)\\
= & \int_{\mathbb{R}_{+}^{n}}\nabla j_{i}(y)\nabla\tilde{\phi}(y)+o(1).
\end{align*}
By definition of $L_{\delta_{m},q_{m}}$ we have 
\begin{multline}
\phi_{\delta_{m},q_{m}}-i^{*}\left(f'(W_{\delta_{m},q_{m}}+\delta_{m}V_{\delta_{m},q_{m}})[\phi_{\delta_{m},q_{m}}]\right)-L_{\delta_{m},q_{m}}\left(\phi_{\delta_{m},q_{m}}\right)\\
=\sum_{i=1}^{n}c_{m}^{i}Z_{\delta_{m},q_{m}}^{i}.\label{eq:L4}
\end{multline}
 We want to prove that, for all $i=1,\dots,n$, $c_{m}^{i}\rightarrow0$
while $m\rightarrow\infty.$ Multiplying equation (\ref{eq:L4}) by
$Z_{\delta_{m},q_{m}}^{k}$ we obtain, by definition (\ref{eq:istella})
of $i^{*}$,
\begin{align*}
\sum_{i=1}^{n}c_{m}^{i}\left\langle \left\langle Z_{\delta_{m},q_{m}}^{i},Z_{\delta_{m},q_{m}}^{k}\right\rangle \right\rangle _{H}= & \left\langle \left\langle i^{*}\left(f'(W_{\delta_{m},q_{m}}+\delta_{m}V_{\delta_{m},q_{m}})[\phi_{\delta_{m},q_{m}}]\right),Z_{\delta_{m},q_{m}}^{k}\right\rangle \right\rangle _{H}\\
= & \int_{\partial M}f'(W_{\delta_{m},q_{m}}+\delta_{m}V_{\delta_{m},q_{m}})[\phi_{\delta_{m},q_{m}}]Z_{\delta_{m},q_{m}}^{k}d\sigma
\end{align*}
Now 
\begin{multline*}
\int_{\partial M}f'(W_{\delta_{m},q_{m}}+\delta_{m}V_{\delta_{m},q_{m}})[\phi_{\delta_{m},q_{m}}]Z_{\delta_{m},q_{m}}^{k}d\sigma\\
=n\int_{\partial M}\left((W_{\delta_{m},q_{m}}+\delta_{m}V_{\delta_{m},q_{m}})^{+}\right)^{\frac{2}{n-2}}\phi_{\delta_{m},q_{m}}Z_{\delta_{m},q_{m}}^{k}d\sigma\\
=n\int_{\partial\mathbb{R}_{n}^{+}}\left((U+\delta_{m}v_{q_{m}})^{+}\right)^{\frac{2}{n-2}}\tilde{\phi}_{m}j_{k}dz+o(1)=n\int_{\partial\mathbb{R}_{n}^{+}}\left(U\right)^{\frac{2}{n-2}}\tilde{\phi}j_{k}dz+o(1)=o(1)
\end{multline*}
since $\tilde{\phi}_{m}\rightharpoonup\tilde{\phi}$ weakly $L^{\frac{2(n-1)}{n-2}}(\partial\mathbb{R}_{+}^{n})$,
$\|v_{q_{m}}\|_{L^{\infty}}$ is bounded independently on $q_{m}$
by (\ref{eq:gradvq}) and by equation (\ref{eq:L3}). At this point,
since
\[
\left\langle \left\langle Z_{\delta_{m},q_{m}}^{i},Z_{\delta_{m},q_{m}}^{j}\right\rangle \right\rangle _{H}=C\delta_{ij}+o(1),
\]
we conclude that $c_{m}^{i}\rightarrow0$ while $m\rightarrow\infty$
for each $i=1,\dots,n$. By (\ref{eq:L4}), and recalling $\|L_{\varepsilon_{m}\lambda_{m},q_{m}}(\phi_{\varepsilon_{m}\lambda_{m},q_{m}})\|_{H}\rightarrow0$
this implies
\begin{multline}
\left\Vert \phi_{\delta_{m},q_{m}}-i^{*}\left(f_{\varepsilon}'(W_{\delta_{m},q_{m}}+\delta_{m}V_{\delta_{m},q_{m}})[\phi_{\delta_{m},q_{m}}]\right)\right\Vert _{H}\\
=\sum_{i=0}^{n-1}c_{m}^{i}\|Z^{i}\|_{H}+o(1)=o(1)\label{eq:L5}
\end{multline}
Now, choose a smooth function $\varphi\in C_{0}^{\infty}(\mathbb{R}_{+}^{n})$
and define 
\[
\varphi_{m}(x)=\frac{1}{\delta_{m}^{\frac{n-2}{2}}}\varphi\left(\frac{1}{\delta_{m}}\left(\psi_{q_{m}}^{\partial}\right)^{-1}(x)\right)\chi\left(\left(\psi_{q_{m}}^{\partial}\right)^{-1}(x)\right)\text{ for }x\in M.
\]
We have that $\|\varphi_{m}\|_{H}$ is bounded and, by (\ref{eq:L5}),
that 
\begin{align*}
\left\langle \left\langle \phi_{\delta_{m},q_{m}},\varphi_{m}\right\rangle \right\rangle _{H}= & \int_{\partial M}f_{\varepsilon_{m}}'(W_{\delta_{m},q_{m}}+\delta_{m}V_{\delta_{m},q_{m}})[\phi_{\delta_{m},q_{m}}]\varphi_{m}d\sigma\\
 & +\left\langle \left\langle \phi_{\delta_{m},q_{m}}-i^{*}\left(f_{\varepsilon_{m}}'(W_{\delta_{m},q_{m}}+\delta_{m}V_{\delta_{m},q_{m}})[\phi_{\delta_{m},q_{m}}]\right),\varphi_{m}\right\rangle \right\rangle _{H}\\
= & \int_{\partial M}f_{\varepsilon_{m}}'(W_{\delta_{m},q_{m}}+\delta_{m}V_{\delta_{m},q_{m}})[\phi_{\delta_{m},q_{m}}]\varphi_{m}d\sigma+o(1)\\
= & n\int_{\partial\mathbb{R}_{+}^{n}}\left((U+\delta_{m}v_{q_{m}})^{+}\right)^{\frac{2}{n-2}}\tilde{\phi}_{m}\varphi dz+o(1)\\
= & n\int_{\mathbb{R}^{n-1}}U^{\frac{2}{n-2}}\tilde{\phi}\varphi dz+o(1),
\end{align*}
by the strong $L_{\text{loc}}^{t}(\partial\mathbb{R}_{+}^{n})$ convergence
of $\tilde{\phi}_{m}$ for $t<\frac{2(n-1)}{n-2}$. On the other hand
\[
\left\langle \left\langle \phi_{\delta_{m},q_{m}},\varphi_{m}\right\rangle \right\rangle _{H}=\int_{\mathbb{R}_{+}^{n}}\nabla\tilde{\phi}\nabla\varphi d\eta+o(1),
\]
so $\tilde{\phi}$ is a weak solution of (\ref{eq:linearizzato})
and we conclude that
\[
\tilde{\phi}\in\text{Span}\left\{ j_{1},\dots j_{n}\right\} .
\]
This, combined with (\ref{eq:L3}) gives that $\tilde{\phi}=0$. Proceeding
as before we have
\begin{align*}
\left\langle \left\langle \phi_{\delta_{m},q_{m}},\phi_{\delta_{m},q_{m}}\right\rangle \right\rangle _{H}= & \int_{\partial M}f_{\varepsilon_{m}}'(W_{\delta_{m},q_{m}}+\delta_{m}V_{\delta_{m},q_{m}})[\phi_{\delta_{m},q_{m}}]\phi_{\delta_{m},q_{m}}d\sigma+o(1)\\
= & n\int_{\partial\mathbb{R}_{+}^{n}}\left((U+\delta_{m}v_{q_{m}})^{+}\right)^{\frac{2}{n-2}}\tilde{\phi}_{m}^{2}dz+o(1)\\
= & n\int_{\partial\mathbb{R}_{+}^{n}}U^{\frac{2}{n-2}}\tilde{\phi}_{m}^{2}dz+o(1)=o(1)
\end{align*}
since $\tilde{\phi}_{m}^{2}$ converges weakly in $L^{\frac{n-1}{n-2}}(\partial\mathbb{R}_{+}^{n})$.
This gives $\left\Vert \phi_{\delta_{m},q_{m}}\right\Vert _{H}\rightarrow0$,
that is a contradiction.
\end{proof}
We have (see \cite[Lemma 9.4 and Lemma 9.5]{A3}) the following relations
\begin{rem}
\label{lem:I-a-m}It holds
\begin{align*}
I_{m}^{\alpha}:=\int_{0}^{\infty}\frac{\rho^{\alpha}}{(1+\rho^{2})^{m}}d\rho=\frac{2m}{\alpha+1}I_{m+1}^{\alpha+2} & \text{ for }\alpha+1<2m\\
I_{m}^{\alpha}=\frac{2m}{2m-\alpha-1}I_{m+1}^{\alpha} & \text{ for }\alpha+1<2m\\
I_{m}^{\alpha}=\frac{2m-\alpha-3}{\alpha+1}I_{m}^{\alpha+2} & \text{ for }\alpha+3<2m.
\end{align*}
In particular we have $I_{n}^{n}=\frac{n-3}{2(n-1)}I_{n-1}^{n}$,
$I_{n-1}^{n-2}=\frac{n-3}{n-1}I_{n-1}^{n}$, $I_{n-2}^{n-2}=\frac{2(n-2)}{n-1}I_{n-1}^{n}$.

Moreover, for $m>k+1$, $m,k\in\mathbb{N}$, we have
\[
\int_{0}^{\infty}\frac{t^{k}}{(1+t)^{m}}dt=\frac{k!}{(m-1)(m-2)\cdots(m-k-1)}
\]
and, by explicit computation, by the previous formula, we obtain:
\[
\int_{\mathbb{R}_{+}^{n}}\frac{dzdt}{\left[(1+t)^{2}+|z|^{2}\right]^{n-1}}=\frac{\omega_{n-1}I_{n-1}^{n-2}}{(n-2)}
\]
\[
\int_{\mathbb{R}_{+}^{n}}\frac{|z|^{2}t^{2}dzdt}{\left[(1+t)^{2}+|z|^{2}\right]^{n}}=\frac{2\omega_{n-1}I_{n}^{n}}{(n-2)(n-3)(n-4)}
\]
\[
\int_{\mathbb{R}_{+}^{n}}\frac{t^{2}dzdt}{\left[(1+t)^{2}+|z|^{2}\right]^{n-1}}=\frac{2\omega_{n-1}I_{n-1}^{n-2}}{(n-2)(n-3)(n-4)}
\]
\[
\int_{\mathbb{R}_{+}^{n}}\frac{|z|^{2}dzdt}{\left[(1+t)^{2}+|z|^{2}\right]^{n-1}}=\frac{\omega_{n-1}I_{n-1}^{n}}{(n-4)}
\]
\end{rem}

\end{document}